\documentclass[12pt]{article} 
\usepackage[colorlinks]{hyperref}
\usepackage{authblk}
\usepackage{graphicx}%
\usepackage{multirow}%
\usepackage{amsmath,amssymb,amsfonts}%
\usepackage{amsthm}%
\usepackage{mathrsfs}%
\usepackage[title]{appendix}%
\usepackage{xcolor}%
\usepackage{textcomp}%
\usepackage{manyfoot}%
\usepackage{booktabs}%
\usepackage{algorithm}%
\usepackage{algorithmicx}%
\usepackage{algpseudocode}%
\usepackage{listings}%
\usepackage{bussproofs}%
\usepackage{multicol}%
\usepackage{enumitem}%
\usepackage[normalem]{ulem}%
\usepackage{mathabx}%
\usepackage[mathscr]{eucal}%
\usepackage{comment}

\theoremstyle{thmstyleone}%
\newtheorem{theorem}{Theorem}
\newtheorem{lemma}[theorem]{Lemma}%
\newtheorem{corollary}[theorem]{Corollary}%

\theoremstyle{thmstyletwo}%
\newtheorem{remark}{Remark}%

\theoremstyle{thmstylethree}%
\newtheorem{definition}{Definition}%
\newtheorem{notation}{Notation}%

\EnableBpAbbreviations

\def\L{\mathcal{L}}
\def\ldl{\mathsf{iK_d}}
\def\LDL{\mathbf{iK_d}}
\def\ldls{\mathsf{iK_{d*}}}
\def\LDLS{\mathbf{iK_{d*}}}

\def\GSTN{\mathbf{STL(N, H)}}
\def\GSTNS{\mathbf{STL(N)}}
\def\LDLSP{\mathbf{iK_{d*}^{+}}}
\def\LDLP{\mathbf{iK_{d}^{+}}}

\def\D{\mathcal{D}}
\def\IH{\mathrm{IH}}
\def\1{\emph{(1)}}
\def\2{\emph{(2)}}
\def\3{\emph{(3)}}
\def\4{\emph{(4)}}

\begin{document}

\title{A Cut-free Sequent Calculus for Basic Intuitionistic Dynamic Topological Logic}

\author[1]{Amirhossein Akbar Tabatabai}
\author[2]{Majid Alizadeh}
\author[2]{Alireza Mahmoudian}
\affil[1]{Bernoulli Institute, University of Groningen}
\affil[2]{School of Mathematics, Statistics and Computer Science, College of Science, University of Tehran}

\date{ }

\maketitle

\begin{abstract}
As part of a broader family of logics, \cite{ImSpace, OGHI} introduced two key logical systems: $\ldl$, which encapsulates the basic logical structure of dynamic topological systems, and $\ldls$, which provides a well-behaved yet sufficiently general framework for an abstract notion of implication. These logics have been thoroughly examined through their algebraic, Kripke-style, and topological semantics.
To complement these investigations with their missing proof-theoretic analysis, this paper introduces a cut-free G3-style sequent calculus for $\ldl$ and $\ldls$. Using these systems, we demonstrate that they satisfy the disjunction property and, more broadly, admit a generalization of Visser's rules. Additionally, we establish that $\ldl$ enjoys the Craig interpolation property and that its sequent system possesses the deductive interpolation property.\\

    
\end{abstract}

\section{Introduction}\label{sec1}

\subsection*{Abstract Implications}

The logical realm encompasses a wide variety of implications, ranging from Boolean and Heyting implications to sub-intuitionistic implications \cite{Vi2, visser1981propositional, Ru}, strict implications in provability logic \cite{visser1981propositional} and preservability logic \cite{Iem1, Iem2, LitViss}, and various types of conditionals \cite{NuteCross}, to name a few.
Efforts have been made to identify an overarching notion of \emph{abstract implication} that unifies some or all of these variants \cite{Celani-Jansana,Lattar}. Here, we adopt the abstract notion of implication proposed in \cite{ImSpace,OGHI}, interpreting implications as the logical mechanism to \emph{internalize} the meta-linguistic provability order $A \vdash B$ between propositions into the proposition $A \to B$ within the formal language itself.\footnote{For a detailed explanation of what we mean by internalization and the justification of the definition of implication, see \cite{ImSpace, OGHI}.}
\begin{definition}\label{dfn: Implication}
Let $\mathsf{A}=(A, \wedge, 1)$ be a bounded meet semi-lattice and $\leq$ be its meet semi-lattice order. A binary order-preserving operation $\to$ from $(A, \leq)^{op} \times (A, \leq)$ to $(A, \leq)$ is called an \emph{implication}, if we have
$a \to a=1$ and
$(a \to b) \wedge (b \to c) \leq a \to c$,
for any $a, b, c \in A$.
\end{definition}
This notion of implication is highly general, and its broad scope proves useful in establishing \emph{negative} results regarding the non-existence of certain types of implication—such as the non-triviality of geometric implications \cite{OnGeometric}. However, compared to Heyting implication, these abstract implications tend to be mathematically ill-behaved \cite{ImSpace, OGHI}. In \cite{ImSpace, OGHI}, it is suggested that the ill-behaved nature of abstract implication stems from the absence of an adjunction—a fundamental feature of Heyting implication. For a comprehensive argument, we refer the reader to \cite{OGHI}. Here, it suffices to emphasize that the adjunction of which Heyting implication is a part is nothing but a complete pair of introduction and elimination rules for the intuitionistic implication. Philosophically, having such a pair means having a complete determination of the proof-theoretic \emph{meaning} of intuitionistic implication, while mathematically, it leads to the elegant proof theory of intuitionistic logic.

To restore this missing adjunction, \cite{ImSpace, OGHI} introduced a new algebraic concept, called a $\nabla$-algebra:
 \begin{definition}\label{dfn:n-alg} 
 A \emph{$\nabla$-algebra} is a structure $(\mathsf{A}, \nabla, \to)$, where $\mathsf{A}$ is a bounded lattice and $\rightarrow$ and $\nabla$ are binary and unary operations on $\mathsf{A}$, respectively, such that:
  \[ 
  \nabla c \wedge a \le b \quad \text{iff} \quad c \le a \rightarrow b \]
 for any $a, b, c \in \mathsf{A}$.
  It follows from this definition that $\nabla$ is monotone and join-preserving. A $\nabla$-algebra is called \emph{normal} if $\nabla$ preserves all \emph{finite} meets. It is called \emph{Heyting}, if $\mathsf{A}$ is a Heyting algebra.
  \end{definition}
$\nabla$-algebras provide a unifying framework that generalizes both bounded lattices ($\nabla a = 0$ and $a \to b = 1$) and Heyting algebras ($\nabla a = a$ and $a \to b = a \supset b$, where $\supset$ denotes the Heyting implication). Additionally, they encompass \emph{temporal Heyting algebras}. To illustrate, in any Heyting $\nabla$-algebra, if we denote the Heyting implication by $\supset$ and define $\Box a$ as $\top \to a$, we obtain the identity $a \to b = \Box (a \supset b)$. Thus, the binary operation $\to$ reduces to the unary operator $\Box$, making a Heyting $\nabla$-algebra effectively a Heyting algebra enriched with the adjunction $\nabla \dashv \Box$. 

This structure naturally lends itself to a temporal interpretation: $\nabla a$ can be understood as ``$a$ was true at some point in the past" while $\Box a$ expresses that ``$a$ will be true at all points in the future." The adjunction $\nabla \dashv \Box$ then corresponds to the well-known relationship between the \emph{existential past} modality and the \emph{universal future} modality. This temporal perspective can be rigorously formulated within the framework of intuitionistic Kripke frames. For the connection with Heyting $\nabla$-algebras, see \cite{OGHI}.

Returning to implications, in \cite{ImSpace, OGHI}, it is proven that in any $\nabla$-algebra, the operator $\to$ is an implication in the sense of Definition \ref{dfn: Implication}, and any abstract implication can be represented as an implication of a normal $\nabla$-algebra, up to a correcting term. Therefore, one can argue that the presence of the adjunction in $\nabla$-algebras ensures a well-behaved family of implications, while the representation theorem guaranties the generality of this well-behaved family. Consequently, \emph{studying normal $\nabla$-algebras offers a mathematical strategy for analyzing all possible implications.} 

\subsection*{Dynamic Topological Systems}

A \emph{dynamic topological system} is a topological space interpreted as the \emph{state space} of a system, augmented with a continuous function that represents the system's \emph{dynamism} \cite{Akin}. To make this non-elementary notion more accessible to logical analysis, several modal approaches have been proposed \cite{artemov1997modal, kremer2004small, Fer}, using either $\mathsf{S4}$ or $\mathsf{IPC}$ to represent the state space, along with some modalities to capture the dynamism. Here, we adopt a basic yet conceptual approach, as outlined below.

First, note that it is beneficial to reformulate the definition of dynamic topological systems in a point-free manner—focusing on open subsets rather than points, since propositions correspond to open sets in a direct manner.
The point-free counterpart of a topological space is a \emph{locale} $\mathscr{X}$, which is a complete lattice (i.e., a lattice that has suprema and infima for all its subsets and is necessarily bounded) where the binary meet operation distributes over arbitrary joins. We can think of this lattice as the lattice of the open subsets of our point-free space. Similarly, the point-free analogue of a continuous function is a \emph{localic map}, which is a function $f^*: \mathscr{Y} \to \mathscr{X}$ between locales that preserves finite meets and arbitrary joins. The localic map  $f^*: \mathscr{Y} \to \mathscr{X}$ can be interpreted as the inverse image of the point-free continuous map $f$ from the point-free space $\mathscr{X}$ to the point-free space $\mathscr{Y}$. Combining these ideas, a \emph{dynamic locale} is a locale $\mathscr{X}$ equipped with a localic map $f^*: \mathscr{X} \to \mathscr{X}$.

This definition is clearly non-elementary, as it refers to arbitrary joins. To make it elementary and hence more logic-friendly, one should replace join-preservation with adjoint operators. Therefore, the fact that binary meet preserves arbitrary joins is transformed into the existence of a right adjoint for the binary meet—a connective known as the Heyting implication. Similarly, the condition that $f^*$ preserves arbitrary joins must be replaced with the right adjoint to $f^*$, denoted by $f_*$. This suggests the following structures as the algebraic and elementary version of a dynamic topological system: $(\mathsf{H}, f^*, f_*)$, where $\mathsf{H}$ is a Heyting algebra, $f^*, f_*: \mathsf{H} \to \mathsf{H}$, with $f^* \dashv f_*$ and $f^*$ preserves all finite meets. This structure is nothing but a normal Heyting $\nabla$-algebra, where $f^*=\nabla$ and $f_*=\Box$. Therefore, \emph{the study of normal Heyting $\nabla$-algebras corresponds to the algebraic study of dynamic topological systems.}

\subsection*{Logics $\ldl$ and $\ldls$}

Motivated by the role of $\nabla$-algebras in representing implications, generalizing (temporal) Heyting algebras, and formalizing dynamic topological systems in an elementary way, \cite{ImSpace,OGHI,OGHII} have studied various algebraic, topological, and logical aspects of $\nabla$-algebras in a series of papers.
In particular, \cite{ImSpace,OGHII} introduced two specific logics, namely the logic $\ldl$ for normal Heyting $\nabla$-algebras and its Heyting implication-free fragment, $\ldls$, for normal $\nabla$-algebras.\footnote{These logics are referred to as $\mathsf{STL(N, H)}$ and $\mathsf{STL(N)}$, respectively in \cite{OGHI}.} The former can be interpreted as the \emph{basic intuitionistic dynamic topological logic}, while the latter serves as a logic for sufficiently general yet well-behaved implications. In \cite{ImSpace,OGHII}, these systems were extensively investigated from algebraic and topological perspectives; however, a proof-theoretic approach remained absent due to the lack of a cut-free sequent calculus for these logics.

In this paper, we aim to fill this gap by introducing the cut-free sequent calculi $\LDL$ and $\LDLS$ for the logics $\ldl$ and $\ldls$, respectively. We then use these systems to demonstrate that both $\ldl$ and $\ldls$ enjoy the disjunction property and, more generally, admit a generalization of Visser's rules. Additionally, we establish that $\ldl$ satisfies the Craig interpolation property and that its sequent system possesses the deductive interpolation property.

The structure of the paper is as follows. In Section \ref{sec: preliminaries}, we recall some preliminary notions and the definition of a sequent system for the logics $\ldl$ and $\ldls$, as presented in \cite{ImSpace}.
In Section \ref{sec: LDL}, we introduce our cut-free sequent systems and investigate their basic properties, including inversion lemmas, the admissibility of contraction, and the deduction theorem. We also prove the equivalence between these systems and those introduced in \cite{ImSpace}, assuming the admissibility of cut in our systems.
In Section \ref{sec: CutAddmissible}, we establish the admissibility of cut as promised. Then, in Section \ref{sec: AdmissibleRules}, we use our cut-free systems to derive the disjunction property and the admissibility of certain generalizations of Visser's rule.
Finally, in Section \ref{sec: Interpolation}, we show that $\ldl$ satisfies the Craig interpolation property and that its sequent system enjoys the deductive interpolation property.\\

\noindent \textbf{Acknowledgments.}
We wish to thank Hiroakira Ono for his helpful suggestions. The first author also gratefully acknowledges the support of the Czech Academy of Sciences (RVO 67985840) and the GA\v{C}R grant 23-04825S.

\section{Preliminaries}\label{sec: preliminaries}

Consider the language $\mathcal{L} = \{\top, \bot, \wedge, \vee, \supset, \to, \nabla\}$, where $\nabla$ is a unary operator, $\supset$ and $\to$ are binary operators, and $\supset$ represents Heyting implication.
We also consider the fragment $\mathcal{L}_* = \mathcal{L} - \{\supset\}$, which excludes Heyting implication.
\emph{$\mathcal{L}$-formulas} and \emph{$\mathcal{L}_*$-formulas}, as well as the corresponding notion of \emph{subformula}, are defined in the usual way. When writing $A \in \mathcal{L}$, we mean that $A$ is an $\mathcal{L}$-formula, and similarly for $\mathcal{L}_*$. Furthermore, unless explicitly stated otherwise, the term ``formula" always refers to an $\mathcal{L}$-formula.
We denote by $V$ the set of all atomic propositions, and by $V(A)$ the set of atoms occurring in the formula $A$.  Additionally, we use the shorthand $\Box A$ to denote $\top \to A$.

Capital Greek letters such as $\Gamma$, $\Sigma$, $\Delta$, $\Pi$, $\Phi$, and $\Lambda$ are used to represent finite multisets of formulas. Capital Roman letters such as $A$, $B$, and $C$ are reserved for individual formulas. Lowercase Roman letters, such as $p$ and $q$, are used for atomic propositions, while $l$, $m$, and $n$ are reserved for natural numbers. Whenever clear from context, we write $A$ to denote the singleton $\{A\}$.
A single comma (``$ , $”) denotes multiset union. Therefore, $\Gamma, A$ represents $\Gamma \cup \{A\}$. To avoid ambiguity, we may explicitly use the symbol $\cup$ when necessary. We write $\Gamma^n$ to denote the $n$-fold union of the multiset $\Gamma$ with itself. For example, $\Gamma, A^2$ stands for $\Gamma \cup \{A, A\}$.
Repeated applications of an operator, such as $\nabla$, on a formula $A$ are denoted by $\nabla^n A$. Similarly, $\nabla^n \Gamma$ represents the multiset $\{\nabla^n A \mid A \in \Gamma\}$. The notation $V(\Gamma)$ refers to $\bigcup_{\gamma \in \Gamma} V(\gamma)$.
Furthermore, $\bigwedge \Gamma$ (resp. $\bigvee \Gamma$) denotes $\bigwedge_{\gamma \in \Gamma} \gamma$ (resp. $\bigvee_{\gamma \in \Gamma} \gamma$). By convention, we set $\bigwedge \varnothing = \top$ and $\bigvee \varnothing = \bot$.

Let $\mathfrak{L} \in \{\mathcal{L}, \mathcal{L}_*\}$. A \emph{sequent over $\mathfrak{L}$} is an expression of the form $S = (\Gamma \Rightarrow \Delta)$, where $\Gamma$ and $\Delta$ are multisets of $\mathfrak{L}$-formulas, and $\Delta$ contains at most one $\mathfrak{L}$-formula. The multiset $\Gamma$ is called \emph{the antecedent} or \emph{the left-side} of $S$, while $\Delta$ is referred to as \emph{the succedent} or \emph{the right-side} of $S$.
A \emph{rule $R$ over $\mathfrak{L}$} is an expression of the form:
\begin{prooftree}
  \AXC{$\{S_i\}_{i \in I}$}
  \RightLabel{\small $R$}
  \UIC{$S$}
\end{prooftree}
where $S_i$'s and $S$ are sequents over $\mathfrak{L}$. The sequents $\{S_i\}_{i \in I}$ are called the \emph{premises}, and the sequent $S$ is referred to as the \emph{conclusion} of $R$.  
By an \emph{axiom}, we mean a rule with no premises, i.e., $I = \varnothing$. Occasionally, we designate a specific formula in the conclusion of a rule as its \emph{principal formula}. This formula is indicated by \uwave{underlining}. An \emph{instance} of a rule is defined in the usual way.

By a \emph{sequent calculus $G$ over $\mathfrak{L}$}, we mean a set of rules over $\mathfrak{L}$. A sequent calculus $G$ is called \emph{analytic} if, in every rule of $G$, the formulas in the premises are subformulas of the formulas in the conclusion. For any set $\mathcal{S} \cup \{S\}$ of sequents, a \emph{proof-tree} for $S$ from the assumptions in $\mathcal{S}$ is defined in the usual way: it is a tree constructed according to the rules in $G$, with the root labeled by $S$ and the leaves labeled by either the axioms of $G$ or sequents from $\mathcal{S}$. Proof-trees are denoted by calligraphic letters such as $\mathcal{D}$ and $\mathcal{D}'$, unless stated otherwise.
The \emph{height} of a proof-tree $\mathcal{D}$ is defined as the length of its longest branch. Note that if $\mathcal{D}$ is an axiom, its height is zero. We denote the height of a proof-tree $\mathcal{D}$ by $h(\mathcal{D})$. Moreover, we write $\mathcal{S} \vdash_h^G S$ to indicate that $S$ has a proof-tree in $G$ from the assumptions in $\mathcal{S}$ with a height of at most $h$. By $\mathcal{S} \vdash_G S$, we mean $\mathcal{S} \vdash_h^G S$ for some $h \in \mathbb{N}$. When $\mathcal{S} = \varnothing$, we denote $\mathcal{S} \vdash_h^G S$ and $\mathcal{S} \vdash_G S$ by  $G \vdash_h S$ and $G \vdash S$, respectively. We say that $G$ proves $S$ if $G \vdash S$. We define the set of all formulas $A$ such that $G \vdash \, \Rightarrow A$ as \emph{the logic of} $G$. 
A rule $R$ is called \emph{admissible} in $G$, if for any instance of $R$, if $G$ proves all the premises of the instance, it also proves its conclusion.

Define the sequent calculus $\GSTN$ over the language $\L$ as the set of rules presented in Figure \ref{GSTN}.
\begin{figure}[h]
\begin{center}
 \begin{tabular}{c c c}
 \AxiomC{}
  \RightLabel{\footnotesize$ Id$}
 \UnaryInfC{$A \Rightarrow A$}
 \DisplayProof 
 &
\small \AxiomC{}
\small \RightLabel{\footnotesize$L \bot$}
\small \UnaryInfC{$ \bot \Rightarrow $}
 \DisplayProof 
&
\small  \AxiomC{}
\small \RightLabel{\footnotesize$R \top$}
\small \UnaryInfC{$ \Rightarrow \top$}
 \DisplayProof
 \\[3ex]
\end{tabular}
 
\begin{tabular}{cc}
\small \AxiomC{$\Gamma \Rightarrow \Delta$}
\small  \RightLabel{\footnotesize$ Lw$}
\small \UnaryInfC{$\Gamma, A \Rightarrow \Delta$}
 \DisplayProof 
 &
\small \AxiomC{$\Gamma \Rightarrow$}
\small \RightLabel{\footnotesize$Rw$}
\small \UnaryInfC{$\Gamma \Rightarrow A $}
 \DisplayProof 
\\[3ex]
 \end{tabular}
\begin{tabular}{cc}
\small  \AxiomC{$\Gamma, A, A \Rightarrow \Delta$}
\small \RightLabel{\footnotesize$Lc$}
\small \UnaryInfC{$\Gamma, A \Rightarrow \Delta$}
 \DisplayProof 
 &
 \small \AxiomC{$\Gamma \Rightarrow A$}
\small \AxiomC{$\Sigma, A \Rightarrow \Delta$}
\small \RightLabel{\footnotesize$cut$} 
 \BinaryInfC{$\Sigma, \Gamma \Rightarrow  \Delta$}
 \DisplayProof
 \\[3ex]
 \end{tabular}
 \begin{tabular}{ccc}
\small \AxiomC{$\Gamma, A \Rightarrow \Delta$}
\small \RightLabel{\footnotesize$L \wedge_1$} 
\small \UnaryInfC{$\Gamma, A \wedge B \Rightarrow \Delta$}
 \DisplayProof 
&
\small \AxiomC{$\Gamma, B \Rightarrow \Delta$}
\small \RightLabel{\footnotesize$L \wedge_2$} 
\small \UnaryInfC{$\Gamma, A \wedge B \Rightarrow \Delta$}
 \DisplayProof
&
\small \AxiomC{$\Gamma \Rightarrow A$}
\small \AxiomC{$\Gamma \Rightarrow B$}
\small \RightLabel{\footnotesize$R \wedge$} 
 \BinaryInfC{$\Gamma \Rightarrow A \wedge B$}
 \DisplayProof
  \\[3ex]
   \end{tabular}
 \begin{tabular}{ccc}
\small \AxiomC{$\Gamma, A \Rightarrow \Delta$}
\small \AxiomC{$\Gamma, B \Rightarrow \Delta$}
\small \RightLabel{\footnotesize$L \vee$} 
 \BinaryInfC{$\Gamma, A \vee B \Rightarrow \Delta$}
 \DisplayProof
 &
\small \AxiomC{$\Gamma \Rightarrow A$}
\small \RightLabel{\footnotesize$R \vee_1$} 
\small \UnaryInfC{$\Gamma \Rightarrow A \vee B$}
 \DisplayProof
&
\small  \AxiomC{$\Gamma \Rightarrow B$}
\small \RightLabel{\footnotesize$R \vee_2$} 
\small \UnaryInfC{$\Gamma \Rightarrow A \vee B$}
 \DisplayProof
 \\[3ex]
\end{tabular}

\begin{tabular}{cc} 
\small \AxiomC{$\Gamma \Rightarrow \Delta$}
\small \RightLabel{\footnotesize$N$} 
\small \UnaryInfC{$\nabla \Gamma \Rightarrow \nabla \Delta$}
 \DisplayProof
 &
\small \AxiomC{$\Gamma \Rightarrow A$}
\small \AxiomC{$\Gamma, B \Rightarrow \Delta$}
\small \RightLabel{\footnotesize$L\!\to$} 
\small \BinaryInfC{$\Gamma, \nabla(A \to B) \Rightarrow \Delta$}
 \DisplayProof

\small \AxiomC{$\nabla \Gamma, A \Rightarrow B$}
\small \RightLabel{\footnotesize$R\!\to$} 
\small \UnaryInfC{$\Gamma \Rightarrow A \to B$}
 \DisplayProof\\[4ex]
\end{tabular}

\begin{tabular}{cc} 
\small \AxiomC{$\Gamma \Rightarrow A$}
\small \AxiomC{$\Gamma, B \Rightarrow \Delta$}
\small \RightLabel{\footnotesize$L\!\supset$} 
\small \BinaryInfC{$\Gamma, A \supset B \Rightarrow \Delta$}
 \DisplayProof
 &
 \small \AxiomC{$\Gamma, A \Rightarrow B$}
\small \RightLabel{\footnotesize$R\!\supset$} 
\small \UnaryInfC{$\Gamma \Rightarrow A \supset B$}
 \DisplayProof
 \\[3ex]
\end{tabular}
\caption{The sequent calculus $\GSTN$}
\label{GSTN}
\end{center}
\end{figure}
Note that the rules $(Id)$, $(L \bot)$, and $(R \top)$ are the axioms of the system. We refer to the rules $(Lw)$, $(Rw)$, $(Lc)$, and $(cut)$ as the \emph{structural rules}, and the rules $(L \wedge_1)$, $(L \wedge_2)$, $(R \wedge)$, $(L \vee)$, $(R \vee_1)$, $(R \vee_2)$, $(L\!\rightarrow)$, $(R\!\rightarrow)$, $(L\!\supset)$, $(R\!\supset)$, and $(N)$ as the \emph{logical rules} of the calculus $\GSTN$. 
The calculus $\GSTNS$ over $\L_*$ is defined similarly to $\GSTN$, except that it does not include the rules $(L\!\supset)$ and $(R\!\supset)$, as $\supset$ is not available in the language $\L_*$. \emph{Basic intuitionistic dynamic topological logic}, denoted by $\ldl$, is the logic of the system $\GSTN$. We also denote the logic of the restricted system $\GSTNS$ by $\ldls$. It is shown in \cite{ImSpace, OGHII} that both $\ldl$ and $\ldls$ are sound and complete with respect to the class of all dynamic topological systems through their canonical interpretation.

The calculus $\GSTNS$, in its generalized substructural form, was first introduced in \cite{ImSpace} and later expanded to $\GSTN$ in \cite{OGHII}, where the corresponding logics were denoted by $\mathsf{STL(N)}$ and $\mathsf{STL(N, H)}$, respectively. The abbreviation $\GSTNS$ (resp. $\GSTN$) stands for Space-Time Logic with Normality (and Heyting implication). In this work, we have chosen to rename the logics as $\ldl$ and $\ldls$ to better emphasize their intuitionistic and dynamic nature, as well as their $K$-style format. 

Note that restricting the language $\mathcal{L}$ to $\mathcal{L}_p = \mathcal{L} - \{\to, \nabla\}$ and using all the rules in Figure \ref{GSTN}, except for $(L\to)$, $(R\to)$, and $(N)$, results in the well-known system $\mathbf{LJ}$. The logic of $\mathbf{LJ}$ is the intuitionistic propositional logic denoted by $\mathsf{IPC}$. Moreover, notice that
in the calculus $\GSTN$, one can simply define the implication $A \to B$ by $\Box (A \supset B)$, effectively reducing the binary operator $\to$ to a unary operator $\Box$. More precisely, let $\mathcal{L}_{\Box}$ be the language obtained by removing the $\to$ operator from $\mathcal{L}$ and adding the unary operator $\Box$, i.e., $\mathcal{L}_{\Box} = \mathcal{L} - \{\to\} \cup \{\Box\}$. Then, define the system $\GSTN_{\Box}$ over the language $\mathcal{L}_{\Box}$ as the calculus consisting of all the rules in Figure \ref{GSTN}, with the modification that the rules $(L\!\to)$ and $(R\!\to)$ are replaced by the following rules:
\vspace{5pt}
\begin{center}
\begin{tabular}{cc} 
\AxiomC{$\Gamma, A \Rightarrow \Delta$}
\RightLabel{\footnotesize$L\Box$} 
\UnaryInfC{$\Gamma, \nabla \Box A \Rightarrow \Delta$}
 \DisplayProof \qquad \quad
&
\AxiomC{$\nabla \Gamma \Rightarrow A$}
\RightLabel{\footnotesize$R\Box$} 
\UnaryInfC{$\Gamma \Rightarrow \Box A$}
 \DisplayProof\\[2ex]
\end{tabular}
\end{center}
By interpreting $A \to B$ as $\Box (A \supset B)$ and $\Box A$ as $\top \to A$, one can easily simulate $\GSTN$ in $\GSTN_{\Box}$ and vice versa. This observation may lead the reader to argue that, to define $\ldl$, it would be preferable to work with the unary operator $\Box$ rather than the binary operator $\to$. While this is indeed a better formalization in cases where the language includes the Heyting implication, the situation is different when working over $\mathcal{L}_*$. In this setting, the operators $\Box$ and $\to$ are not necessarily inter-definable. Since our primary interest lies in implications represented by normal $\nabla$-algebras, we opt to use $\to$ in both languages to maintain a unified and simpler presentation.

\section{A $\nabla$-analytic system for $\ldl$ and $\ldls$} \label{sec: LDL}

The calculus $\GSTN$ is not analytic. Besides the usual issue with the $(cut)$ rule, another problematic rule is $(R\!\rightarrow)$, where the formulas in $\nabla \Gamma$ may not be subformulas of any formula in the conclusion. This extra $\nabla$ in the premise of $(R\!\rightarrow)$ can be addressed by relaxing the analyticity criterion to consider subformulas ``up to some occurrences of $\nabla$", which would still allow proof-theoretic methods to apply. However, as is standard, we still require to eliminate the $(cut)$ rule—an impossible task in this setting.
For example, the sequent $\nabla(p \vee q) \Rightarrow \nabla p \vee \nabla q$ is provable in $\GSTN$ (see Lemma \ref{lem:l-nabla-dist}), yet it has no cut-free proof, given the structure of the rules in $\GSTN$. 

To address this issue, we introduce two cut-free sequent calculi, $\LDL$ and $\LDLS$, and aim to establish that they correspond to the logics $\ldl$ and $\ldls$, respectively. To this end, we will investigate several properties of these calculi, including inversion lemmas and the admissibility of the contraction rule. These results will allow us to demonstrate the admissibility of $(cut)$ in these calculi, thereby proving their equivalence to $\GSTN$ and $\GSTNS$, respectively. Ultimately, this confirms that $\ldl$ and $\ldls$ are the logics of the calculi $\LDL$ and $\LDLS$, respectively.

Define $\LDL$ as the sequent calculus over $\L$ presented in Figure \ref{LDL}.
\begin{figure}[h]
\begin{center}
 \begin{tabular}{c c c}
\footnotesize \AxiomC{}
\footnotesize  \RightLabel{\footnotesize$ Id^p$}
\footnotesize \UnaryInfC{$p \Rightarrow p$}
 \DisplayProof 
 &
\footnotesize \AxiomC{}
\footnotesize \RightLabel{\footnotesize$L \bot$}
\footnotesize \UnaryInfC{$ \uwave{\bot} \Rightarrow $}
 \DisplayProof 
&
\footnotesize \AxiomC{}
\footnotesize \RightLabel{\footnotesize$R \top$}
\footnotesize \UnaryInfC{$ \Rightarrow \uwave{\top}$}
 \DisplayProof
 \\[3ex]
\end{tabular}
 
\begin{tabular}{cc}
\footnotesize \AxiomC{$\Gamma \Rightarrow \Delta$}
\footnotesize \RightLabel{\footnotesize$ LW$}
\footnotesize \UnaryInfC{$\Gamma, \Sigma \Rightarrow \Delta$}
 \DisplayProof 
 &
\footnotesize \AxiomC{$\Gamma \Rightarrow$}
\footnotesize \RightLabel{\footnotesize$Rw$}
\footnotesize \UnaryInfC{$\Gamma \Rightarrow A $}
 \DisplayProof 
\\[4ex]
 \end{tabular}
 \begin{tabular}{cc}
\footnotesize \AxiomC{$\Gamma, \nabla^n A, \nabla^n B \Rightarrow \Delta$}
\footnotesize \RightLabel{\footnotesize$L \wedge^n$} 
\footnotesize \UnaryInfC{$\Gamma, \uwave{\nabla^n(A \wedge B)} \Rightarrow \Delta$}
 \DisplayProof
 &
\footnotesize \AxiomC{$\Gamma \Rightarrow A$}
\footnotesize \AxiomC{$\Gamma \Rightarrow B$}
\footnotesize \RightLabel{\footnotesize$R \wedge$} 
 \BinaryInfC{$\Gamma \Rightarrow \uwave{A \wedge B}$}
 \DisplayProof
  \\[5ex]
\end{tabular}

\begin{tabular}{ccc}
\footnotesize \AxiomC{$\Gamma, \nabla^n A \Rightarrow \Delta$}
\footnotesize \AxiomC{$\Gamma, \nabla^n B \Rightarrow \Delta$}
\footnotesize \RightLabel{\footnotesize$L \vee^n$} 
 \BinaryInfC{$\Gamma, \uwave{\nabla^n(A \vee B)} \Rightarrow \Delta$}
 \DisplayProof
 &
\footnotesize \AxiomC{$\Gamma \Rightarrow A$}
\footnotesize \RightLabel{\footnotesize$R \vee_1$} 
\footnotesize \UnaryInfC{$\Gamma \Rightarrow \uwave{A \vee B}$}
 \DisplayProof
&
\footnotesize \AxiomC{$\Gamma \Rightarrow B$}
\footnotesize \RightLabel{\footnotesize$R \vee_2$} 
\footnotesize \UnaryInfC{$\Gamma \Rightarrow \uwave{A \vee B}$}
 \DisplayProof
 \\[5ex]
\end{tabular}

\begin{tabular}{cc} 
\footnotesize \AxiomC{$\Gamma, \nabla^{n+1}(A \to B) \Rightarrow \nabla^n A$}
\footnotesize \AxiomC{$\Gamma, \nabla^{n+1}(A \to B), \nabla^n B \Rightarrow \Delta$}
\footnotesize \RightLabel{\footnotesize$L\!\to^n$} 
\footnotesize \BinaryInfC{$\Gamma, \uwave{\nabla^{n+1}(A \to B)} \Rightarrow \Delta$}
 \DisplayProof 
& 
\footnotesize \AxiomC{$\nabla \Gamma, A \Rightarrow B$}
\footnotesize \RightLabel{\footnotesize$R\!\to$} 
\footnotesize \UnaryInfC{$\Gamma \Rightarrow \uwave{A \to B}$}
 \DisplayProof\\[5ex]
\end{tabular}

\begin{tabular}{cc} 
\footnotesize \AxiomC{$\Gamma, \nabla^n(A \supset B) \Rightarrow \nabla^n A$}
\footnotesize \AxiomC{$\Gamma, \nabla^n B \Rightarrow \Delta$}
\footnotesize \RightLabel{\footnotesize$L\!\supset^n$} 
\footnotesize \BinaryInfC{$\Gamma, \uwave{\nabla^n(A \supset B)} \Rightarrow \Delta$}
 \DisplayProof
 &
\footnotesize \AxiomC{$\Gamma, A \Rightarrow B$}
\footnotesize \RightLabel{\footnotesize$R\!\supset$} 
\footnotesize \UnaryInfC{$\Gamma \Rightarrow \uwave{A \supset B}$}
 \DisplayProof
 \\[5ex]
\end{tabular}

\begin{tabular}{c} 
\footnotesize \AxiomC{$\Gamma \Rightarrow \Delta$}
\footnotesize \RightLabel{\footnotesize$N$} 
\footnotesize \UnaryInfC{$\nabla \Gamma \Rightarrow \nabla \Delta$}
 \DisplayProof\\[3ex]
\end{tabular}
\caption{The sequent calculus $\LDL$. In the axiom $(Id^p)$, the letter $p$ is an atom. }
\label{LDL}
\end{center}
\end{figure}
Note that the rules $(Id^p)$, $(L \bot)$, and $(R \top)$ are the axioms of $\LDL$. We refer to the rules $(L \wedge^n)$, $(R \wedge)$, $(L \vee^n)$, $(R \vee_1)$, $(R \vee_2)$, $(L \!\rightarrow^n)$, $(R\! \rightarrow)$, $(L \!\supset^n)$, and $(R\!\supset)$ as the \emph{logical rules}, and $(LW)$ and $(Rw)$ as the \emph{structural rules} of the calculus $\LDL$. Note that if a rule does not have an underlined formula, it does not have a principal formula. Therefore, not all rules have a principal formula.
We also define $\LDLS$ over the language $\L_*$ as the sequent calculus consisting of all the axioms and rules in Figure \ref{LDL}, except the rules $(L\!\supset^n)$ and $(R\!\supset)$. Recall that the connective $\supset$ does not appear in the language $\L_*$.

\begin{remark}\label{rem:ldl-gstn}
Observe the distinctions between $\GSTN$ and $\LDL$. In $\LDL$, the $(cut)$ and $(Lc)$ rules are omitted; its $(Id^p)$ axiom is a simplified version of $(Id)$, restricted to atomic formulas, and its $(LW)$ rule allows introducing multiple formulas in the conclusion—a process achievable in $\GSTN$ through repeated applications of $(Lw)$.
Furthermore, $\LDL$ differs in the rules $(L \wedge^n)$, $(L \vee^n)$, $(L \!\rightarrow^n)$, and $(L \!\supset^n)$, which permit an arbitrary number of $\nabla$'s on the principal formulas, distributing them across immediate subformulas in the premises. The form of these premises also differs slightly. For instance, the rule $(L \wedge^0)$ includes both $A$ and $B$ in its premise, while in $\GSTN$, $(L \wedge_1)$ and $(L \wedge_2)$ each involve only one of them. However, by employing $(LW)$, the rules $(L \wedge_1)$, $(L \wedge_2)$, $(L\vee)$, $(L \!\rightarrow)$, and $(L \!\supset)$ in $\GSTN$ are readily derivable in $\LDL$.
Additionally, note that the rules $(R \wedge)$, $(R \vee)$, and $(R \supset)$, along with $(L \wedge^0)$, $(L \vee^0)$, and $(L \!\supset^0)$, correspond to the standard propositional rules in the system $\mathbf{G3i}$ as defined in \cite{troelstra2000basic}.
\end{remark}

\subsection{Some Properties of $\LDL$ and $\LDLS$}
In this subsection, we prove some fundamental properties of the systems $\LDL$ and $\LDLS$, as outlined earlier. To begin, we demonstrate the provability of some useful sequents within these systems, which will be required for subsequent discussions.

\begin{lemma}\label{lem:id-adm}
Let $\Gamma$, $\Sigma$ and $\Delta$ be multisets of formulas and $A$ and $B$ be formulas. Then:
  \begin{description}
     \item[$(i)$]
     $\LDL \vdash A \Rightarrow A$.
    \item[$(ii)$]
    $\LDL \vdash A , \nabla (A \rightarrow B) \Rightarrow B$.
   \item[$(iii)$] 
    $\Gamma, A \Rightarrow \Delta \vdash_{\LDL} \Gamma, \nabla \Box A \Rightarrow \Delta$.
     \item[$(iv)$] 
    $\Gamma \Rightarrow A \vdash_{\LDL} \Box \Gamma \Rightarrow \Box A$.
     \item[$(v)$] 
     $\nabla \Gamma, \Sigma, A \Rightarrow B \vdash_{\LDL} \Gamma, \Box \Sigma \Rightarrow A \rightarrow B$.
  \end{description}
  All claims also hold for the language $\L_*$ and the system $\LDLS$.
\end{lemma}
\begin{proof}
	We only prove the case for $\LDL$. The other is identical. For $(i)$, we proceed by induction on the structure of the formula $A$. If $A$ is an atom, $\bot$, or $\top$, we obtain $(A \Rightarrow A)$ by using the axioms $(Id^p)$, $(L\bot)$, and $(R\top)$, along with the rules $(LW)$ or $(Rw)$ of $\LDL$. If $A = B \circ C$, where $\circ \in \{\wedge, \vee, \supset\}$, the claim follows directly from the induction hypothesis for $B$ and $C$. We will provide a detailed explanation of the proof for the cases $A = B \rightarrow C$ and $A = \nabla B$.

  For $A = B \rightarrow C$, let $\D_B$ and $\D_C$ be the proof-trees obtained from the induction hypothesis for $B$ and $C$, respectively. We now consider the following proof-tree:
\begin{prooftree}
  \AXC{$\D_B$} \noLine
  \UIC{$B \Rightarrow B$} \RightLabel{$LW$}
  \UIC{$\nabla (B \rightarrow C), B \Rightarrow B$}
  \AXC{$\D_C$} \noLine
  \UIC{$C \Rightarrow C$} \RightLabel{$LW$}
  \UIC{$\nabla (B \rightarrow C), C \Rightarrow C$}
  \RightLabel{$L \!\rightarrow^0$}
  \BIC{$\nabla (B \rightarrow C), B \Rightarrow C$}
  \RightLabel{$R \!\rightarrow$}
  \UIC{$B \rightarrow C \Rightarrow B \rightarrow C$}
\end{prooftree}

For $A = \nabla B$, we can prove $(\nabla B \Rightarrow \nabla B)$ by applying the rule $(N)$ to the proof of $(B \Rightarrow B)$ obtained from the induction hypothesis.

 For $(ii)$, consider the following proof-tree:
    \begin{prooftree}
      \AXC{}
      \RightLabel{$(i)$}
      \UIC{$A \Rightarrow A$}
      \RightLabel{$LW$}
      \UIC{$A, \nabla (A \rightarrow B) \Rightarrow A$}
    
      \AXC{}
      \RightLabel{$(i)$}
      \UIC{$B \Rightarrow B$}
      \RightLabel{$LW$}
      \UIC{$A, \nabla (A \rightarrow B), B \Rightarrow B$}
    
      \RightLabel{$L\!\rightarrow^0$}
      \BIC{$A , \nabla (A \rightarrow B) \Rightarrow B$}
    \end{prooftree} 
where $(i)$ means using part $(i)$. For $(iii)$, consider the following proof-tree:
 \begin{prooftree}
 \AXC{$ $}
  \RightLabel{$R\top$}
    \UIC{$\Rightarrow \top$}
    \RightLabel{$LW$}
    \UIC{$\Gamma, \nabla \Box A \Rightarrow \top$}
    \AXC{$\Gamma, A \Rightarrow \Delta$}
    \RightLabel{$LW$}
    \UIC{$\Gamma, \nabla \Box A, A \Rightarrow \Delta$}
    \RightLabel{$L\!\rightarrow^0$}
    \BIC{$\Gamma, \nabla \Box A \Rightarrow \Delta$}
  \end{prooftree}
For $(iv)$, consider the following proof-tree:
\begin{prooftree}
        \AXC{$\Gamma \Rightarrow A$}
        \RightLabel{$(iii)$} \doubleLine
        \UIC{$\nabla \Box \Gamma \Rightarrow A$}
        \RightLabel{$LW$}
        \UIC{$\nabla \Box \Gamma, \top \Rightarrow A$}
        \RightLabel{$R \!\rightarrow$}
        \UIC{$\Box \Gamma \Rightarrow \Box A$}
      \end{prooftree}
where by $(iii)$, we mean repeated applications of part $(iii)$.  
Finally, for $(v)$, consider the following proof-tree:
\begin{prooftree}
        \AXC{$\nabla \Gamma, \Sigma, A \Rightarrow B$}
        \RightLabel{$(iii)$} \doubleLine
        \UIC{$\nabla \Gamma, \nabla \Box \Sigma, A \Rightarrow B$}
        \RightLabel{$R\!\rightarrow$}
        \UIC{$\Gamma, \Box \Sigma \Rightarrow A \rightarrow B$}
      \end{prooftree}
where again $(iii)$ means repeated applications of part $(iii)$.
\end{proof}

The next property to consider is \emph{$\nabla$-analyticity}.
The systems $\LDL$ and $\LDLS$ are not analytic. For example, in the premise of the rule $(R\!\rightarrow)$, the formulas in $\nabla \Gamma$ are not necessarily subformulas of any formula in the conclusion. Similarly, the formulas $\nabla^n A$ and $\nabla^n B$ in the rule $(L \wedge^n)$ are not necessarily subformulas of the conclusion, if $n \geq 1$. 
However, we can describe these systems as \emph{$\nabla$-analytic}, i.e., in every premise, each formula is of the form $\nabla^n A$ for some $n \geq 0$, where $A$ is a subformula of a formula in the conclusion. In other words, every formula in any premise is a subformula of a formula in the conclusion, up to additional $\nabla$ operators. 

$\nabla$-analyticity captures nearly all of the power associated with analyticity. For instance, a direct implication of the $\nabla$-analyticity of $\LDL$ is the \emph{conservativity} of the calculus $\LDL$ over the calculus $\LDLS$:

\begin{lemma}\label{obs:ldl-cons-ext}
$\LDLS \vdash_h \Gamma \Rightarrow \Delta$ iff $\LDL \vdash_h \Gamma \Rightarrow \Delta$, for any sequent $\Gamma \Rightarrow \Delta$ over $\L_*$.
\end{lemma}
\begin{proof}
The left-to-right direction is straightforward, as $\LDL$ includes $\LDLS$. For the converse, we leverage the $\nabla$-analyticity of $\LDL$. Since the sequent $\Gamma \Rightarrow \Delta$ contains no occurrence of the connective $\supset$, this connective cannot appear anywhere in its proof. Consequently, the proof does not involve the rules $(L \!\supset^n)$ or $(R \!\supset)$. As a result, the derivation remains entirely within the system $\LDLS$, where these rules are not present.
\end{proof}

By applying Lemma \ref{obs:ldl-cons-ext}, many properties can be established by first proving them for $\LDL$, after which the conservativity of $\LDL$ over $\LDLS$ ensures their validity in $\LDLS$ as well. This approach will be utilized repeatedly in the remainder of this section.

\begin{lemma}\label{lem:arrow-bot}
The following hold:
\begin{description}
    \item[$(i)$]
    $\LDL \vdash \, \Rightarrow A \rightarrow B$ if and only if $\LDL \vdash A \Rightarrow B$, for any $A, B \in \mathcal{L}$.
    \item[$(ii)$] 
    $\LDL \not \vdash \{ A_i \to B_i \}_{i \in I} \Rightarrow \,$, for any finite set $I$ and any family $\{A_i, B_i\}_{i \in I} \subseteq \mathcal{L}$. 
\end{description}
The same also holds for $\LDLS$.
\end{lemma}
\begin{proof}
We only prove the case for $\LDL$. The other is a consequence of Lemma \ref{obs:ldl-cons-ext}. For $(i)$, if $A \Rightarrow B$ is provable, proving $\Rightarrow A \to B$ is straightforward by applying the rule $(R\!\to)$. For the converse, if $\Rightarrow A \to B$ is provable by the proof-tree $\mathcal{D}$, then an easy inspection of the form of the rules in $\LDL$ reveals that the last rule of $\mathcal{D}$ must be $(R\!\to)$, which implies the provability of $A \Rightarrow B$.

For $(ii)$, by induction on $\mathcal{D}$, we prove that $\mathcal{D}$ cannot be a proof-tree for a sequent in the form $\{ A_i \to B_i \}_{i \in I} \Rightarrow \,$ in $\LDL$. Clearly $\mathcal{D}$ cannot be an axiom. Hence, the base case trivially holds. For the induction step, the only possible last rule for $\mathcal{D}$ is $(LW)$. Hence, there is a shorter proof of $\mathcal{D}$ for the sequent $\{A_i \to B_i\}_{i \in I'} \Rightarrow$, for some $I' \subseteq I$. This is impossible by the induction hypothesis.
\end{proof}

The following is the \emph{inversion lemma} for certain rules of $\LDL$ and $\LDLS$. We will use this lemma to prove the admissibility of $(Lc)$ in $\LDL$ and $\LDLS$.

\begin{lemma}[Inversion]\label{lem:inv} For any natural numbers $h$ and $n$, multisets  $\Gamma$ and $\Delta$ of formulas and formulas $A$ and $B$, we have the following:
	\begin{description}
		\item[$(i)$]
        If $\LDL \vdash_h \Gamma, \nabla^n (A \wedge B) \Rightarrow \Delta$ then $\LDL \vdash_h \Gamma, \nabla^n A, \nabla^n B \Rightarrow \Delta$.
		\item[$(ii)$]
        If $\LDL \vdash_h \Gamma, \nabla^n (A \vee B) \Rightarrow \Delta$ then $\LDL \vdash_h \Gamma, \nabla^n A \Rightarrow \Delta$ and $\LDL \vdash_h \Gamma, \nabla^n B \Rightarrow \Delta$.
  	\item[$(iii)$]
        If $\LDL \vdash_h \Gamma, \nabla^n (A \supset B) \Rightarrow \Delta$ then $\LDL \vdash_h \Gamma, \nabla^n B \Rightarrow \Delta$.
	\end{description}
The parts $(i)$ and $(ii)$ also hold for $\LDLS$.
\end{lemma}
\begin{proof}
We only prove the case for $\LDL$. The other is a consequence of Lemma \ref{obs:ldl-cons-ext}. We will explain the proof for $(iii)$; the other two parts are similar. For $(iii)$, we use an induction on $h$. It is clear that $h$ cannot be $0$, as $\Gamma, \nabla^n (A \supset B) \Rightarrow \Delta$ does not match the form of any axiom and thus cannot be derived as one. Therefore, $h > 0$ which means that the proof of $\Gamma, \nabla^n (A \supset B) \Rightarrow \Delta$ ends with a non-axiomatic rule. We now investigate each of these rules. If the proof-tree ends with the rule $(L \!\supset^n)$ and $\nabla^n (A \supset B)$ is the principal formula, the claim follows directly. If the last rule is $(LW)$ and $\nabla^n (A \supset B)$ is in the newly added formulas by the weakening, the proof is straightforward. In any other case, simply commute the last rule with the sequent(s) derived from the induction hypothesis.
Here, we will focus on the cases for $(N)$ and $(R \!\rightarrow)$, as these are the only cases where the context on the left changes:\\

\noindent $\bullet$ Suppose the last rule is $(N)$. Therefore, $n \geq 1$, and the last rule of the proof-tree of height at most $h > 0$ is in the form:
  \begin{prooftree}
    \AXC{$\Gamma', \nabla^{n-1} (A \supset B) \Rightarrow \Delta'$} \RightLabel{$N$}
    \UIC{$\nabla \Gamma', \nabla^n (A \supset B) \Rightarrow \nabla \Delta'$}
  \end{prooftree}
where $\Gamma = \nabla \Gamma'$ and $\Delta = \nabla \Delta'$, for some multisets $\Gamma'$ and $\Delta'$. By the induction hypothesis for the premise, we have a proof-tree for $\Gamma', \nabla^{n-1} B \Rightarrow \Delta'$ of height at most $h-1$. Using $(N)$, we then obtain a proof-tree for $\nabla \Gamma', \nabla^n B \Rightarrow \nabla \Delta'$, which has height at most $h$.\\

\noindent $\bullet$ Suppose the last rule is $(R \!\rightarrow)$. Therefore, $\Delta = \{C \rightarrow D\}$, for some formulas $C$ and $D$, and the last rule of the proof-tree of height at most $h > 0$ is in the form:
  \begin{prooftree}
    \AXC{$\nabla \Gamma, \nabla^{n+1} (A \supset B), C \Rightarrow D$}
    \RightLabel{$R \!\rightarrow$}
    \UIC{$\Gamma, \nabla^n (A \supset B) \Rightarrow C \rightarrow D$}
  \end{prooftree}
By the induction hypothesis, we have a proof-tree for $\nabla \Gamma, \nabla^{n+1} B, C \Rightarrow D$ of height at most $h-1$. By applying the rule $(R \!\rightarrow)$, we obtain a proof-tree for $\Gamma, \nabla^n B \Rightarrow C \rightarrow D$ of height at most $h$.
\end{proof}

\begin{theorem}[Height-preserving admissibility of $(Lc)$]\label{thm:lc-adm} If $\LDL \vdash_h \Gamma, A, A \Rightarrow \Delta$ then $\LDL \vdash_h \Gamma, A \Rightarrow \Delta$, for any multisets $\Gamma$ and $\Delta$ of formulas, any formula $A$ and any $h \geq 0$. Restricting to $\mathcal{L}_*$-formulas, the same also holds for $\LDLS$.
\end{theorem}
\begin{proof}
We only prove the case for $\LDL$. The other is a consequence of Lemma \ref{obs:ldl-cons-ext}. For $\LDL$, we use induction on $h$. For $h = 0$, the sequent $\Gamma, A, A \Rightarrow \Delta$ must be an axiom, which is impossible upon inspecting the form of the axioms. For $h > 0$, we investigate different cases based on the last rule of the proof of $\Gamma, A, A \Rightarrow \Delta$. If the last rule is $(LW)$ or $(N)$, the claim follows easily. The cases where none of the occurrences of $A$ are principal are also straightforward, as it suffices to apply the same rule to the proof-trees provided by the induction hypothesis.
Therefore, suppose one occurrence of $A$ is principal, and the last rule is either $(L \wedge^n)$, $(L \vee^n)$, $(L \!\rightarrow^n)$, or $(L \!\supset^n)$. In this case, we have the following scenarios:\\

		\noindent $\bullet$ For $(L \wedge ^n)$, we have $A = \nabla^n (B \wedge C)$, for some formulas $B$ and $C$ and the proof is in the form:

		\begin{prooftree}
			\AXC{$\D_0$} \noLine
			\UIC{$\Gamma, \nabla^n B, \nabla^n C, \nabla^n (B \wedge C) \Rightarrow \Delta$}
			\RightLabel{$L \wedge ^n$}
			\UIC{$\Gamma, \nabla^n (B \wedge C), \nabla^n (B \wedge C) \Rightarrow \Delta$}		
		\end{prooftree}
Then, by Lemma \ref{lem:inv}, we obtain a proof for $\Gamma, \nabla^n B, \nabla^n B, \nabla^n C, \nabla^n C \Rightarrow \Delta$ of height at most $h-1$. Using the induction hypothesis twice, we then obtain a proof for $\Gamma, \nabla^n B, \nabla^n C \Rightarrow \Delta$ of height at most $h-1$. Finally, applying $(L \wedge^n)$ to this proof results in the desired proof with height at most $h$.\\
	
		\noindent $\bullet$ For $(L \vee^n)$, we have $A = \nabla^n (B \vee C)$, for some formulas $B$ and $C$ and the proof is in the form:
		\begin{prooftree}
			\AXC{$\D_0$} \noLine
			\UIC{$ \Gamma, \nabla^n B, \nabla^n (B \vee C) \Rightarrow \Delta$}
			\AXC{$\D_1$} \noLine
			\UIC{$\Gamma, \nabla^n C, \nabla^n (B \vee C) \Rightarrow \Delta$}
			\RightLabel{$L \vee ^n$}
			\BIC{$ \Gamma, \nabla^n (B \vee C), \nabla^n (B \vee C) \Rightarrow \Delta$}		
		\end{prooftree}
		By Lemma \ref{lem:inv} for $\D_0$ and $\D_1$, we obtain proofs for $\Gamma, \nabla^n B, \nabla^n B \Rightarrow \Delta$ and $\Gamma, \nabla^n C, \nabla^n C \Rightarrow \Delta$ of heights at most $h-1$. By the induction hypothesis, we also have proofs for $\Gamma, \nabla^n B \Rightarrow \Delta$ and $\Gamma, \nabla^n C \Rightarrow \Delta$ of heights at most $h-1$. Finally, by applying $(L \vee ^n)$ to these proofs, we obtain the desired proof with height at most $h$.\\
	
		\noindent $\bullet$ For $(L \!\rightarrow^n)$, we have $A = \nabla^{n+1} (B \rightarrow C)$, for some formulas $B$ and $C$. The proof of $\Gamma, \nabla^{n+1} (B \rightarrow C), \nabla^{n+1} (B \rightarrow C) \Rightarrow \Delta$ has the premises 
        \[
        \Gamma, \nabla^{n+1} (B \rightarrow C), \nabla^{n+1} (B \rightarrow C) \Rightarrow \nabla^n B
        \]
        and 
        \[
        \Gamma, \nabla^n C, \nabla^{n+1} (B \rightarrow C), \nabla^{n+1} (B \rightarrow C) \Rightarrow \Delta.
        \]
        By the induction hypothesis, we have proofs for $\Gamma, \nabla^{n+1} (B \rightarrow C) \Rightarrow \nabla^n B$ and $\Gamma, \nabla^n C, \nabla^{n+1} (B \rightarrow C) \Rightarrow \Delta$ with height at most $h-1$. Finally, by applying the rule $(L \!\rightarrow^n)$ to these proofs, we obtain the desired proof with height at most $h$.\\

		\noindent $\bullet$ For $(L\!\supset^n)$, we have $A = \nabla^n (B \supset C)$, for some formulas $B$ and $C$ and the last rule is in the form:
		\begin{prooftree}
			\AXC{$\D_0$} \noLine
			\UIC{$\Gamma, \nabla^n (B \supset C), \nabla^n (B \supset C) \Rightarrow \nabla^n B$}
			\AXC{$\D_1$} \noLine
			\UIC{$\Gamma, \nabla^n C, \nabla^n (B \supset C) \Rightarrow \Delta$}
			\RightLabel{$L\!\supset^n$}
			\BIC{$\Gamma, \nabla^n (B \supset C), \nabla^n (B \supset C) \Rightarrow \Delta$}
		\end{prooftree}
		By Lemma \ref{lem:inv} for $\D_1$, we obtain a proof for $\Gamma, \nabla^n C, \nabla^n C \Rightarrow \Delta$ of height at most $h-1$. By the induction hypothesis, we have proofs for $\Gamma, \nabla^n (B \supset C) \Rightarrow \nabla^n B$ and $\Gamma, \nabla^n C \Rightarrow \Delta$, both with heights at most $h-1$. Finally, by applying the rule $(L \!\supset^n)$, we obtain the desired proof with height at most $h$.
\end{proof}

\begin{theorem}\label{thm:cut-adm}\emph{(Admissibility of cut)}
  If $\LDL \vdash \Gamma \Rightarrow A$ and $\LDL \vdash \Sigma, A \Rightarrow \Delta$, then $\LDL \vdash \Gamma, \Sigma \Rightarrow \Delta$. The same also holds for $\LDLS$.
\end{theorem}
\begin{proof}
We postpone the proof to Section \ref{sec: CutAddmissible}.
\end{proof}

\begin{notation}
Thanks to Lemma \ref{lem:id-adm} and Theorems \ref{thm:lc-adm} and \ref{thm:cut-adm}, we can now assume that the axiom $(Id)$ and the rules $(Lc)$ and $(cut)$ are available in $\LDL$ and $\LDLS$. From this point onward, we will use these rules in the proof-trees of $\LDL$ and $\LDLS$ without further explanation.
\end{notation}

\subsection{Equivalence of $\GSTN$ and $\LDL$}

In this subsection, we aim to show that $\LDL$ (resp. $\LDLS$) and $\GSTN$ (resp. $\GSTNS$) are equivalent, meaning they prove precisely the same set of sequents. This will establish that the logics $\ldl$ and $\ldls$ are the logics of $\LDL$ and $\LDLS$, respectively. 
To this end, we first need to establish the admissibility of the rules $(L \wedge ^n)$, $(L \vee ^n)$, $(L \!\rightarrow^n)$, and $(L \!\supset^n)$ of $\LDL$ (resp. $\LDLS$) within $\GSTN$ (resp. $\GSTNS$).

\begin{lemma}\label{lem:l-nabla-dist} For any $n \geq 0$ and any formulas $A, B \in \L$, we have the following:
	\begin{description}
		\item[$(i)$] 
        $\GSTN \vdash \nabla^n (A \vee B) \Rightarrow \nabla^n A \vee \nabla^n B$.
	
		\item[$(ii)$]
        $\GSTN \vdash \nabla^n (A \wedge B) \Rightarrow \nabla^n A \wedge \nabla^n B$. 
	
		\item[$(iii)$] 
        $\GSTN \vdash \nabla^n (A \rightarrow B) \Rightarrow \nabla^n A \rightarrow \nabla^n B$.
		
        \item[$(iv)$]
        $\GSTN \vdash \nabla^n (A \supset B) \Rightarrow \nabla^n A \supset \nabla^n B$.
	\end{description}
  Moreover, $(i)$, $(ii)$ and $(iii)$ also hold for $\GSTNS$, when $A, B \in \L_*$.
\end{lemma}
\begin{proof}
We will only prove the case for $\GSTN$, as the other case follows identically. For $(i)$, let us first consider the following proof-tree in $\GSTN$:
  \begin{prooftree}
    \AXC{}
    \RightLabel{$R \top$}
    \UIC{$\Rightarrow \top$}
    \RightLabel{$Lw$}
    \UIC{$\nabla (\top \rightarrow C) \Rightarrow \top$}
  
    \AXC{}
    \RightLabel{$Id$}
    \UIC{$C \Rightarrow C$}
    \RightLabel{$Lw$}
    \UIC{$\nabla (\top \rightarrow C), C \Rightarrow C$}
  
    \RightLabel{$L \!\rightarrow$}
    \BIC{$\nabla (\top \rightarrow C) \Rightarrow C$}
  \end{prooftree}
Now, substitute $C = \nabla A \vee \nabla B$, call the resulting proof-tree $\mathcal{D}$ and consider the following proof-tree:
  \begin{prooftree}
    \AXC{}
    \RightLabel{$Id$}
    \UIC{$\nabla A \Rightarrow \nabla A$}
    \RightLabel{$R\vee_1$}
    \UIC{$\nabla A \Rightarrow \nabla A \vee \nabla B$}
    \RightLabel{$Lw$}
    \UIC{$\nabla A , \top \Rightarrow \nabla A \vee \nabla B$}
    \RightLabel{$R \!\rightarrow$}
    \UIC{$A \Rightarrow \top \rightarrow (\nabla A \vee \nabla B)$}
  
    \AXC{}
    \RightLabel{$Id$}
    \UIC{$\nabla B \Rightarrow \nabla B$}
    \RightLabel{$R\vee_2$}
    \UIC{$\nabla B \Rightarrow \nabla A \vee \nabla B$}
    \RightLabel{$Lw$}
    \UIC{$\nabla B , \top \Rightarrow \nabla A \vee \nabla B$}
    \RightLabel{$R \!\rightarrow$}
    \UIC{$B \Rightarrow \top \rightarrow (\nabla A \vee \nabla B)$}
  
    \RightLabel{$L\vee$}
    \BIC{$A \vee B \Rightarrow \top \rightarrow (\nabla A \vee \nabla B)$}
    \RightLabel{$N$}
    \UIC{$\nabla (A \vee B) \Rightarrow \nabla (\top \rightarrow (\nabla A \vee \nabla B))$}
  
    \AXC{$\D$}
  
    \RightLabel{$cut$}
    \BIC{$\nabla (A \vee B) \Rightarrow \nabla A \vee \nabla B$}
  \end{prooftree}
Call this proof-tree $\D_1$. Construct $\D_{n+1}$ inductively as follows:
 {\small
  \begin{prooftree}
    \AXC{$\D_n$}
    \noLine
    \UIC{$\nabla^n (A \vee B) \Rightarrow \nabla^n A \vee \nabla^n B$}
    \RightLabel{$N$}
    \UIC{$\nabla^{n+1} (A \vee B) \Rightarrow \nabla (\nabla^n A \vee \nabla^n B)$}
  
    \AXC{$\D_1$}
    \noLine
    \UIC{$\nabla (\nabla^n A \vee \nabla^n B) \Rightarrow \nabla^{n+1} A \vee \nabla^{n+1} B$}
    
    \RightLabel{$cut$}
    \BIC{$\nabla^{n+1} (A \vee B) \Rightarrow \nabla^{n+1} A \vee \nabla^{n+1} B$}
  \end{prooftree}
  }
This complete the proof of $(i)$.
  For $(ii)$, $(iii)$ and $(iv)$, consider the following proof-trees:
  \begin{prooftree}
    \AXC{}
    \RightLabel{$Id$}
    \UIC{$A \Rightarrow A$}
    \RightLabel{$L \wedge_1$}
    \UIC{$A \wedge B \Rightarrow A$}
    \RightLabel{$N^{(*)}$} \doubleLine
    \UIC{$\nabla^n (A \wedge B) \Rightarrow \nabla^n A$}
  
    \AXC{}
    \RightLabel{$Id$}
    \UIC{$B \Rightarrow B$}
    \RightLabel{$L \wedge_2$}
    \UIC{$A \wedge B \Rightarrow B$}
    \RightLabel{$N^{(*)}$} \doubleLine	
    \UIC{$\nabla^n (A \wedge B) \Rightarrow \nabla^n B$}
    
    \RightLabel{$R \wedge$}
    \BIC{$\nabla^n (A \wedge B) \Rightarrow \nabla^n A \wedge \nabla^n B$}
  \end{prooftree}

  \begin{prooftree}
    \AXC{}
    \RightLabel{$Id$}
    \UIC{$A \Rightarrow A$}
  
    \AXC{}
    \RightLabel{$Id$}
    \UIC{$B \Rightarrow B$}
    \RightLabel{$Lw$}
    \UIC{$A, B \Rightarrow B$}
  
    \RightLabel{$L \!\rightarrow$}
    \BIC{$\nabla (A \rightarrow B) , A \Rightarrow B$}
    \RightLabel{$N^{(*)}$} \doubleLine
    \UIC{$\nabla^{n+1} (A \rightarrow B) , \nabla^n A \Rightarrow \nabla^n B$}
    \RightLabel{$R \!\rightarrow$}
    \UIC{$\nabla^n (A \rightarrow B) \Rightarrow \nabla^n A \rightarrow \nabla^n B$}
  \end{prooftree}

  \begin{prooftree}
    \AXC{}
    \RightLabel{$Id$}
    \UIC{$A \Rightarrow A$}

    \AXC{}
    \RightLabel{$Id$}
    \UIC{$B \Rightarrow B$}
    \RightLabel{$Lw$}
    \UIC{$A, B \Rightarrow B$}
  
    \RightLabel{$L \!\supset$}
    \BIC{$A \supset B , A \Rightarrow B$}
    \RightLabel{$N^{(*)}$} \doubleLine
    \UIC{$\nabla^n (A \supset B) , \nabla^n A \Rightarrow \nabla^n B$}
    \RightLabel{$R \!\supset$}
    \UIC{$\nabla^n (A \supset B) \Rightarrow \nabla^n A \supset \nabla^n B$}
  \end{prooftree}
where $N^{(*)}$ denotes multiple uses of the rule $N$.
\end{proof}

\begin{theorem}\label{thm:ldl-eq-gstn}
The following hold:
    \begin{description}
      \item[$(i)$]
      $\GSTN \vdash \Gamma \Rightarrow \Delta$ iff $\LDL \vdash \Gamma \Rightarrow \Delta$, for any sequent $\Gamma \Rightarrow \Delta$ over $\L$.
      \item[$(ii)$]
       $\GSTNS \vdash \Gamma \Rightarrow \Delta$ iff $\LDLS \vdash \Gamma \Rightarrow \Delta$, for any sequent $\Gamma \Rightarrow \Delta$ over $\L_*$.
  \end{description}
  \end{theorem}
  \begin{proof}
 We prove only $(i)$, as the proof of $(ii)$ follows similarly.
   For $(i)$, by induction on the proof-tree for $\Gamma \Rightarrow \Delta$, the claim is easy to prove in both directions. For the left-to-right direction, as observed in Remark \ref{rem:ldl-gstn}, it suffices to show the provability of the axiom $(Id)$ and the admissibility of the rules $(Lc)$ and $(cut)$ in $\LDL$. These are established in Lemma \ref{lem:id-adm}, Theorem \ref{thm:lc-adm}, and Theorem \ref{thm:cut-adm}, respectively. For the other direction, we essentially use Lemma \ref{lem:l-nabla-dist} whenever needed. We will explain the case where the proof-tree for $\Gamma \Rightarrow \Delta$ in $\LDL$ ends with the rule $(L \!\rightarrow^n)$. The other cases are similar. In this case, by the induction hypothesis, we have proof-trees $\D_1$ and $\D_2$ for $\Gamma, \nabla^{n+1} (A \rightarrow B) \Rightarrow \nabla^n A$ and $\Gamma, \nabla^{n+1} (A \rightarrow B), \nabla^n B \Rightarrow \Delta$ in $\GSTN$. We also have a proof-tree $\D'_0$ in $\GSTN$ for $\nabla^n (A \rightarrow B) \Rightarrow \nabla^n A \rightarrow \nabla^n B$, by Lemma \ref{lem:l-nabla-dist} part $(iii)$. Now, consider the following proof-trees $\D'$ and $\D_3$:
  
    \begin{prooftree}
      \AXC{$\D'_0$}
      \noLine
      \UIC{$\nabla^n (A \rightarrow B) \Rightarrow \nabla^n A \rightarrow \nabla^n B$}
      \RightLabel{$N$} \LeftLabel{$\D' \quad$}
      \UIC{$\nabla^{n+1} (A \rightarrow B) \Rightarrow \nabla (\nabla^n A \rightarrow \nabla^n B)$}
    \end{prooftree}
  
    \begin{prooftree}
      \AXC{$\D_1$}
      \noLine
      \UIC{$\Gamma, \nabla^{n+1} (A \rightarrow B) \Rightarrow \nabla^n A$}

      \AXC{$\D_2$}
      \noLine
      \UIC{$\Gamma, \nabla^{n+1} (A \rightarrow B), \nabla^n B \Rightarrow \Delta$}
      \RightLabel{$L \!\rightarrow$} \LeftLabel{$\D_3 \quad$}
      \BIC{$\Gamma, \nabla^{n+1} (A \rightarrow B), \nabla (\nabla^n A \rightarrow \nabla^n B) \Rightarrow \Delta$}
    \end{prooftree}
By applying the cut rule on $\D'$ and $\D_3$, and performing one application of $(Lc)$, we obtain the desired sequent in $\GSTN$:
    \begin{prooftree}
      \AXC{$\D'$}
      \AXC{$\D_3$}
      \RightLabel{$cut$}
      \BIC{$\Gamma, (\nabla^{n+1} (A \rightarrow B))^2 \Rightarrow \Delta$}
      \RightLabel{$Lc$}
      \UIC{$\Gamma, \nabla^{n+1} (A \rightarrow B) \Rightarrow \Delta$}
    \end{prooftree}
  \end{proof}

\begin{corollary}\label{LogicAndCalculi}
  $\ldl$ and $\ldls$ are the logics of $\LDL$ and $\LDLS$, respectively. Moreover, $\ldl$ is conservative over $\ldls$.
\end{corollary}
\begin{proof}
Since $\ldl$ (resp. $\ldls$) is the logic of $\GSTN$ (resp. $\GSTNS$), the first part directly follows from Theorem \ref{thm:ldl-eq-gstn}. The second part then follows from the conservativity of $\LDL$ over $\LDLS$ (Lemma \ref{obs:ldl-cons-ext}) combined with the first part.
\end{proof}

\subsection{Deduction Theorem} 
Let $\mathbf{LJ}^+$ be the calculus $\mathbf{LJ}$ extended with the cut rule, and recall the \emph{deduction theorem} stating the equivalence between $\Rightarrow A \vdash_{\mathbf{LJ}^+} \, \Rightarrow B$ and $\vdash_{\mathbf{LJ}} A \Rightarrow B$. This form of the deduction theorem does not hold for $\LDL$ and $\LDLS$. For example, using the rule $(N)$, it is easy to see that $\Rightarrow p \vdash_{\LDL} \, \Rightarrow \nabla p$. However, by examining the form of the rules of $\LDL$, it becomes clear that the sequent $p \Rightarrow \nabla p$ cannot be provable in $\LDL$. In this section, we present a modified version of the deduction theorem that holds for $\LDL$ and $\LDLS$. First, we need the following definition:

\begin{definition}
  The set of \emph{variants} of a formula $A$, is defined inductively as follows:
  \begin{description}
      \item[$(i)$]
      $A$ is a variant of $A$, and
     \item[$(ii)$]
     if $B$ is a variant of $A$, then $\nabla B$ and $\Box B$ are also variants of $A$.
  \end{description}
  To put it simply, the variants of a formula $A$ are all possible combinations of $\nabla$ and $\Box$ applied to $A$.
\end{definition}

Let $\LDLP$ (resp. $\LDLSP$) be the calculus $\LDL$ (resp. $\LDLS$) extended by the cut rule. Note that the cut admissibility in $\LDL$ (resp. $\LDLS$) only shows that the provable sequents in $\LDLP$ (resp. $\LDLSP$) and $\LDL$ (resp. $\LDLS$) are the same if the proofs have no assumptions. As usual, in the presence of assumptions, these systems are different, and the task of any deduction theorem is to connect the provability of these two systems in such a general situation:
\begin{theorem}\label{thm:deduction}(Deduction Theorem)
  For any multisets $\Gamma$ and $\Delta$ of formulas and any formula $A$, the following are equivalent:
  \begin{description}
      \item[$(i)$] 
$\Rightarrow A \vdash_{\LDLP} \Gamma \Rightarrow \Delta$.
        \item[$(ii)$] 
$\LDL \vdash \Gamma, \Sigma_A \Rightarrow \Delta$, for a set $\Sigma_A$ of the variants of $A$.     
  \end{description}
A similar claim also holds, replacing $\LDL$ with $\LDLS$ and $\LDLP$ with $\LDLSP$.
\end{theorem}
\begin{proof}
We only prove the case for $\LDL$. The other is identical. Note that in $(ii)$, we can also use the cut rule in the proofs in $\LDL$ as it is admissible and we have no assumptions. To prove $(ii)$ from $(i)$, we use an induction on a proof-tree $\mathcal{D}$ of $\, \Rightarrow A \vdash_{\LDLP} \Gamma \Rightarrow \Delta$.
If $\Gamma \Rightarrow \Delta$ is an instance of the axioms of $\LDL$, the claim is trivial by setting $\Sigma_A = \varnothing$. If $\Gamma \Rightarrow \Delta$ is the sequent $\Rightarrow A$, itself, then set $\Sigma_A = \{ A \}$. The claim is proved by Lemma \ref{lem:id-adm}, part $(i)$.
If $\mathcal{D}$ is not an axiom, it ends with one of the rules in $\LDL$ or the rule $(cut)$. There are two cases to consider: either the last rule is in $\{(N), (R\!\to)\}$ or it is not:\\

 \noindent $\bullet$ $(N)$: As $\D$ ends with $(N)$, we have $\Gamma=\nabla \Gamma'$ and $\Delta = \nabla \Delta'$, for some multisets $\Gamma'$ and $\Delta'$. Therefore, $\D$ is in the form:
  \begin{prooftree}
    \AXC{$\D'$} \noLine
    \UIC{$\Gamma' \Rightarrow \Delta'$} 
    \RightLabel{$N$}
    \UIC{$\nabla \Gamma' \Rightarrow \nabla \Delta'$}
  \end{prooftree}
By the induction hypothesis, there is a set $\Sigma'_A$ of the variants of $A$ and a proof-tree $\D_2$ to witness $\LDL \vdash \Gamma', \Sigma'_A \Rightarrow \Delta'$. Now, consider the following proof-tree:
  \begin{prooftree}
    \AXC{$\D_2$} \noLine
    \UIC{$\Sigma'_A, \Gamma' \Rightarrow \Delta'$}
   \RightLabel{$N$}
    \UIC{$\nabla \Sigma'_A, \nabla \Gamma' \Rightarrow \nabla \Delta'$}
  \end{prooftree}
Define $\Sigma_A=\nabla \Sigma'_A$. It is clear that $\Sigma_A$ is also a set of the variants of $A$ and $\LDL \vdash \Gamma, \Sigma_A \Rightarrow \Delta$.\\

\noindent $\bullet$ $(R \!\rightarrow$): As $\D$ ends with $(R \!\rightarrow)$, we have $\Delta = \{B \rightarrow C\}$, for some formulas $B$ and $C$, and $\D$ has the following form:
  \begin{prooftree}
    \AXC{$\D'$} \noLine
    \UIC{$\nabla \Gamma, B \Rightarrow C$}
    \RightLabel{$R \!\rightarrow$}
    \UIC{$\Gamma \Rightarrow B \rightarrow C$}
  \end{prooftree}
  By the induction hypothesis, there is a set $\Sigma'_A$ of the variants of $A$ such that $\LDL \vdash \Sigma'_A, \nabla \Gamma, B \Rightarrow C$. Then, by Lemma \ref{lem:id-adm}, Part $(v)$, we have $\LDL \vdash \Box \Sigma'_A, \Gamma \Rightarrow B \to C$.
Define $\Sigma_A=\Box \Sigma'_A$. It is clear that $\Sigma_A$ is also a set of the variants of $A$ and $\LDL \vdash \Gamma, \Sigma_A \Rightarrow \Delta$. \\

In the second case, the strategy is as follows: Apply the induction hypothesis to the premise(s) of the last rule to find $\Sigma_A'$ (resp. $\Sigma_A'$ and $\Sigma_A''$) if the rule has one premise (resp. two premises). Then, it is enough to set $\Sigma_A$ as $\Sigma_A'$ (resp. $\Sigma_A' \cup \Sigma_A''$). Using $(LW)$ and then the rule itself, it is easy to prove the claim. We explain this strategy for the rule $(L\!\rightarrow^n)$, and the case for the rest, including the $(cut)$ rule, follows similarly.\\

\noindent $\bullet$ $(L\!\to^n)$: As $\D$ ends with $(L \!\rightarrow^n)$, we have $\Gamma = \Gamma', \nabla^{n+1} (B \rightarrow C)$, for some multiset $\Gamma'$ and some formulas $B$ and $C$. Thus, $\D$ is in the form:
  \begin{prooftree}
    \noLine
    \AXC{$\D'$}
    \UIC{$\Gamma', \nabla^{n+1} (B \rightarrow C) \Rightarrow \nabla^n B$}
    \noLine
    \AXC{$\D''$}
    \UIC{$\Gamma', \nabla^{n+1} (B \rightarrow C), \nabla^n C \Rightarrow \Delta$}
    \RightLabel{$L \!\rightarrow ^n$}
    \BIC{$\Gamma', \nabla^{n+1} (B \rightarrow C) \Rightarrow \Delta$}
  \end{prooftree}
By the induction hypothesis, there are sets $\Sigma_A'$ and $\Sigma_A''$ of the variants of $A$ such that $\LDL \vdash \Sigma_A', \Gamma', \nabla^{n+1} (B \rightarrow C) \Rightarrow \nabla^n B$ and 
\[
\LDL \vdash \Sigma_A'', \Gamma', \nabla^{n+1} (B \rightarrow C), \nabla^n C \Rightarrow \Delta.
\]
Let $\D'_2$ and $\D''_2$ be the corresponding proof-trees, respectively. Now, consider the following proof-tree:
  { \footnotesize
  \begin{prooftree}
    \noLine
    \AXC{$\D_2'$}
    \UIC{$\Sigma'_A, \Gamma', \nabla^{n+1} (B \rightarrow C) \Rightarrow \nabla^n B$}
    \RightLabel{$LW$} 
    \UIC{$\Sigma'_A, \Sigma''_A, \Gamma', \nabla^{n+1} (B \rightarrow C) \Rightarrow \nabla^{n} B$}

    \noLine
    \AXC{$\D_2''$}
    \UIC{$\Sigma''_A, \Gamma', \nabla^{n+1} (B \rightarrow C), \nabla^n C \Rightarrow \Delta$}
    \RightLabel{$LW$} 
    \UIC{$\Sigma'_A, \Sigma''_A, \Gamma', \nabla^{n+1} (B \rightarrow C), \nabla^n C \Rightarrow \Delta$}
    \RightLabel{$L \!\rightarrow ^{n}$}
    \BIC{$\Sigma'_A, \Sigma''_A, \Gamma', \nabla^{n+1} (B \rightarrow C) \Rightarrow \Delta$}
  \end{prooftree}
  } 
\noindent Setting $\Sigma_A=\Sigma'_A \cup \Sigma''_A$, we proved the claim.

To prove $(i)$ from $(ii)$, first, using the rule $(N)$ and Lemma \ref{lem:id-adm} Part $(iv)$, it is clear that for any variant $B$ of $A$, we have $\, \Rightarrow A \vdash_\LDL \, \Rightarrow B$. Now, suppose $\LDL \vdash \Gamma, \Sigma_A \Rightarrow \Delta$. Then, $\, \Rightarrow A \vdash_\LDL \, \Rightarrow B$, for any $B \in \Sigma_A$. Therefore, by applying the cut rule, we reach $\, \Rightarrow A \vdash_{\LDLP} \Gamma \Rightarrow \Delta$.
\end{proof}

\section{Admissibility of $cut$ in $\LDL$}\label{sec: CutAddmissible}
In this section, we establish the admissibility of the cut rule in $\LDL$ by proving Theorem \ref{thm:cut-adm}. Using Lemma \ref{obs:ldl-cons-ext}, this also implies the cut admissibility in $\LDLS$. For technical reasons, we prove the admissibility of the following more general rule:
 \begin{prooftree}
     \AXC{$\Pi \Rightarrow A$}
     \AXC{$\Phi, \nabla^n A \Rightarrow \Lambda$}
     \BIC{$\Phi, \nabla^n \Pi \Rightarrow \Lambda$}
   \end{prooftree}
where $n \geq 0$ is an arbitrary natural number. Clearly, in the special case $n = 0$, the rule reduces to the standard cut rule. To establish the admissibility of this generalized cut rule, we first need to define the rank of a formula.

\begin{definition}
  The rank of a formula $A$, denoted by $\rho(A)$, is defined recursively as follows:
  \[ \rho(A) = \begin{cases}
  1 & \quad  A \in V \cup \{ \bot, \top \} \\
  \rho(B) & \quad  A = \nabla B \\
  max(\rho(B), \rho(C)) + 1 & \quad  A = B \circ C \quad (\circ \in \{ \land, \lor, \rightarrow, \supset \})
  \end{cases} \]
  Notice that $\nabla$ does not increase the rank of a formula.
\end{definition}


\begin{theorem}(Generalized cut admissibility)
  If $\LDL \vdash \Pi \Rightarrow A$ and $\LDL \vdash \Phi, \nabla^n A \Rightarrow \Lambda$, then $\LDL \vdash \Phi, \nabla^n \Pi \Rightarrow \Lambda$.
\end{theorem}
\begin{proof}
Consider the lexicographic well-ordering on $\mathbb{N}^2$. By induction on this order, we show that for any pair $(r, h)$ of natural numbers, any formula $A$ of rank $r$ and any proof-trees $\D_1$ and $\D_2$ for
$\Pi \Rightarrow A$ and $\Phi, \nabla^n A \Rightarrow \Lambda$, respectively, if $h(\D_1) + h(\D_2) \leq h$, then there is a proof in $\LDL$ for $\Phi, \nabla^n \Pi \Rightarrow \Lambda$.

Let us prove the claim for $(r, h)$ and let $\D_1$ and $\D_2$ be proof-trees for
$\Pi \Rightarrow A$ and $\Phi, \nabla^n A \Rightarrow \Lambda$, respectively. For simplicity, we use the following notation. Let $\D_1'$ and $\D_2'$ be proof-trees for $\Pi' \Rightarrow A'$ and $\Phi', \nabla^{n'} A' \Rightarrow \Lambda'$ such that either $\rho(A') < \rho(A)$ or $\rho(A') = \rho(A)$ but $h(\D_1') + h(\D_2') <  h(\D_1) + h(\D_2)$. Then, as the induction hypothesis is applicable on $\D_1'$ and $\D_2'$, we can have a proof for $\Phi', \nabla^{n'} \Pi' \Rightarrow \Lambda'$. We denote this proof by:

\begin{prooftree}
  \AXC{$\D_1'$} \noLine
  \UIC{$\Pi' \Rightarrow A'$}
  \AXC{$\D_2'$} \noLine
  \UIC{$\Phi', \nabla^{n'} A' \Rightarrow \Lambda'$}
  \RightLabel{$\IH_{A'}$}
  \BIC{$ \Phi', \nabla^{n'} \Pi' \Rightarrow \Lambda'$}
\end{prooftree}
To prove the claim for $(r, h)$, we investigate all possible last rules for $\D_1$ and $\D_2$, and we partition all these cases into three parts:
  \begin{enumerate}
    \item The cases where $A$ is not principal in the last rule of $\D_1$. 

    \item The cases where $A$ is principal in the last rule of $\D_1$, but $\nabla^n A$ is not principal in the last rule of $\D_2$.

    \item The cases where $A$ is principal in the last rule of $\D_1$ and $\nabla^n A$ is principal in the last rule of $\D_2$.
  \end{enumerate}
In the first two parts, the general strategy is to apply the induction hypothesis to the immediate subtrees of $\mathcal{D}_1$ and $\mathcal{D}_2$, using the cut formula with the same rank as $A$. In the final part, however, the induction hypothesis is also applied to formulas with ranks smaller than that of $A$. Below, we provide a detailed explanation of the key cases in each part.

  Part \1. It is clear that $\mathcal{D}_1$ must either be one of the axioms $(Id^p)$ or $(L \bot)$, or it must end with one of the following rules: $(Rw)$, $(LW)$, $(L \wedge^r)$, $(L \vee^r)$, $(L \!\rightarrow^r)$, $(L \!\supset^r)$, or $(N)$. The case where $\mathcal{D}_1$ ends with $(Id^p)$ is straightforward. The case for $(L \bot)$ is impossible, as the succedent of $\Pi \Rightarrow A$ must be non-empty. If $\mathcal{D}_1$ ends with $(Rw)$, then $\Pi \Rightarrow$ is provable. Thus, using $(N)$, $(LW)$, and $(Rw)$, one can readily derive $\Phi, \nabla^n \Pi\Rightarrow \Delta$. For the remaining cases, we proceed with the following argument:\\

  \noindent $\bullet$ ($L W$): As $\D_1$ ends with $(LW)$, we have $\Pi = \Pi_1, \Pi_2$, for some multisets $\Pi_1$ and $\Pi_2$. Then $\D_1$ is in the form:
  \begin{prooftree}
		\noLine
		\AXC{$\D_1'$}
		\UIC{$\Pi_1 \Rightarrow A$}
		
		\RightLabel{$LW$}
		\UIC{$\Pi_1, \Pi_2 \Rightarrow A$}
  \end{prooftree}
  Now, consider the following proof-tree:
  \begin{prooftree}
    \noLine
    \AXC{$\D_1'$}
    \UIC{$\Pi_1 \Rightarrow A$}

    \noLine
    \AXC{$\D_2$}
    \UIC{$\Phi, \nabla^n A \Rightarrow \Lambda$}

    \RightLabel{$\IH_A$}
    \BIC{$\Phi, \nabla^n \Pi_1 \Rightarrow \Lambda$}

    \RightLabel{$LW$}
    \UIC{$\Phi, \nabla^n \Pi_1, \nabla^{n} \Pi_2 \Rightarrow \Lambda$}
  \end{prooftree}
which is what we wanted. \quad\\
  
  \noindent $\bullet$ ($L \wedge^r$): As $\D_1$ ends with $(L \wedge^r)$, we have $\Pi = \Gamma, \nabla^r (B \wedge C)$, for some multiset $\Gamma$ and some formulas $B$ and $C$. Then $\D_1$ is in the form:
  \begin{prooftree}
		\noLine
		\AXC{$\D_1'$}
		\UIC{$\Gamma, \nabla^r B, \nabla^r C \Rightarrow A$}
		
		\RightLabel{$L \wedge ^r$}
		\UIC{$\Gamma, \nabla^r (B \wedge C) \Rightarrow A$}
  \end{prooftree}
  Now, consider the following proof-tree:
  \begin{prooftree}
    \noLine
    \AXC{$\D_1'$}
    \UIC{$\Gamma, \nabla^r B, \nabla^r C \Rightarrow A$}

    \noLine
    \AXC{$\D_2$}
    \UIC{$\Phi, \nabla^n A\Rightarrow \Lambda$}

    \RightLabel{$\IH_A$}
    \BIC{$\Phi, \nabla^n \Gamma, \nabla^{n+r} B, \nabla^{n+r} C \Rightarrow \Lambda$}

    \RightLabel{$L \wedge ^{n+r}$}
    \UIC{$\Phi, \nabla^n \Gamma, \nabla^{n+r} (B \wedge C) \Rightarrow \Lambda$}
  \end{prooftree}
which is what we wanted. \quad\\

  \noindent $\bullet$ $(L\vee^r)$: As $\D_1$ ends with $(L \vee^r)$, we have $\Pi = \Gamma, \nabla^r (B \vee C)$, for some multiset $\Gamma$ and some formulas $B$ and $C$, and $\D_1$ is in the form:
	\begin{prooftree}
    \noLine
    \AXC{$\D_1'$}
    \UIC{$\Gamma, \nabla^r B \Rightarrow A$}
    
    \noLine
    \AXC{$\D_1''$}
    \UIC{$\Gamma, \nabla^r C \Rightarrow A$}
    
    \RightLabel{$L \vee ^r$}
    \BIC{$\Gamma, \nabla^r (B \vee C) \Rightarrow A$}
	\end{prooftree}
	Now, consider the following proof-tree:
	\begin{prooftree}
		\noLine
		\AXC{$\D_1'$}
		\UIC{$\Gamma, \nabla^r B \Rightarrow A$}
		
		\noLine
		\AXC{$\D_2$}
		\UIC{$\Phi, \nabla^n A \Rightarrow \Lambda$}
		
		\RightLabel{$\IH_A$}
		\BIC{$\Phi, \nabla^n \Gamma, \nabla^{n+r} B \Rightarrow \Lambda$}

		\noLine
		\AXC{$\D_1''$}
		\UIC{$\Gamma, \nabla^r C \Rightarrow A$}
		
		\noLine
		\AXC{$\D_2$}
		\UIC{$\Phi, \nabla^n A \Rightarrow \Lambda$}
		
		\RightLabel{$\IH_A$}
		\BIC{$\Phi, \nabla^n \Gamma, \nabla^{n+r} C \Rightarrow \Lambda$}

		\RightLabel{$L \vee ^{n+r}$}
		\BIC{$\Phi, \nabla^n \Gamma, \nabla^{n+r} (B \vee C) \Rightarrow \Lambda$}
	\end{prooftree}
which is what we wanted. \quad\\

  \noindent $\bullet$ $(L\!\to^r)$: As $\D_1$ ends with $(L \!\rightarrow^r)$, we have $\Pi = \Gamma, \nabla^{r+1} (B \rightarrow C)$, for some multiset $\Gamma$ and some formulas $B$ and $C$. Thus, $\D_1$ is in the form:
  \begin{prooftree}
    \noLine
    \AXC{$\D_1'$}
    \UIC{$\Gamma, \nabla^{r+1} (B \rightarrow C) \Rightarrow \nabla^r B$}
    \noLine
    \AXC{$\D_1''$}
    \UIC{$\Gamma, \nabla^{r+1} (B \rightarrow C), \nabla^r C \Rightarrow A$}
    \RightLabel{$L \!\rightarrow ^r$}
    \BIC{$\Gamma, \nabla^{r+1} (B \rightarrow C) \Rightarrow A$}
  \end{prooftree}
  Consider the following proof-tree:
  { \scriptsize
  \begin{prooftree}
    \noLine
    \AXC{$\D_1'$}
    \UIC{$\Gamma, \nabla^{r+1} (B \rightarrow C) \Rightarrow \nabla^r B$}
    \RightLabel{$N^{(*)}$} \doubleLine
    \UIC{$\nabla^n \Gamma, \nabla^{n+r+1} (B \rightarrow C) \Rightarrow \nabla^{n+r} B$}
    \RightLabel{$LW$}
    \UIC{$\Phi, \nabla^n \Gamma, \nabla^{n+r+1} (B \rightarrow C) \Rightarrow \nabla^{n+r} B$}

    \noLine
    \AXC{$\D_1''$}
    \UIC{$\Gamma, \nabla^{r+1} (B \rightarrow C), \nabla^r C \Rightarrow A$}
    \noLine
    \AXC{$\D_2$}
    \UIC{$\Phi, \nabla^n A \Rightarrow \Lambda$}
    \RightLabel{$\IH_A$}
    \BIC{$\Phi, \nabla^n \Gamma, \nabla^{n+r+1} (B \rightarrow C), \nabla^{n+r} C \Rightarrow \Lambda$}

    \RightLabel{$L \!\rightarrow ^{n+r}$}
    \BIC{$\Phi, \nabla^n \Gamma, \nabla^{n+r+1} (B \rightarrow C) \Rightarrow \Lambda$}
  \end{prooftree}
  } 
  \noindent where $N^{(*)}$ means multiple use of the rule $N$.  \quad\\

  \noindent $\bullet$ The case $(L\!\supset^r)$ is handled similarly to the case $(L\!\to^r)$. \quad\\

  \noindent $\bullet$ $(N)$: As $\D_1$ ends with $(N)$, we have $\Pi=\nabla \Gamma$ and $A = \nabla B$, for some multiset $\Gamma$ and some formula $ B$. Therefore, $\D_2$ proves $\Phi, \nabla^n (\nabla B) \Rightarrow \Lambda$ and $\D_1$ is in the form:
  \begin{prooftree}
    \AXC{$\D_1'$} \noLine
    \UIC{$\Gamma \Rightarrow B$} \RightLabel{$N$}
    \UIC{$\nabla \Gamma \Rightarrow \nabla B$}
  \end{prooftree}
  Consider the following proof-tree:
  \begin{prooftree}
    \AXC{$\D_1'$} \noLine
    \UIC{$\Gamma \Rightarrow B$}
    \AXC{$\D_2$} \noLine
    \UIC{$\Phi, \nabla^{n+1} B \Rightarrow \Lambda$}
    \RightLabel{$\IH_B$}
    \BIC{$\Phi, \nabla^{n+1} \Gamma \Rightarrow \Lambda$}
  \end{prooftree}
which is what we wanted. Note that $\rho(A) = \rho(B)$, and since the height of $\mathcal{D}'_1$ is smaller than that of $\mathcal{D}_1$, the induction hypothesis can be applied. \quad\\

  Part \2. In this part, we consider the cases where $\nabla^n A$ is not principal in the last rule of $\mathcal{D}_2$, while $A$ is principal in the last rule of $\mathcal{D}_1$. Consequently, $\mathcal{D}_1$ is either the axiom $(R\top)$ or it ends with one of the following rules: $(R \wedge)$, $(R \vee_1)$, $(R \vee_2)$, $(R \!\rightarrow)$, or $(R \!\supset)$. Note that this implies $A$ is neither an atom nor a formula of the form $\nabla B$. In this part, we fix $\mathcal{D}_1$ and apply the induction hypothesis to proof-trees with smaller height than $\mathcal{D}_2$.

  First, observe that $\mathcal{D}_2$ cannot be the axiom $(Id^p)$, because if it were, then $\nabla^n A$ would be an atom. This implies $n = 0$ and $A$ is an atom, which contradicts the fact that $A$ cannot be an atom. Second, note that $\mathcal{D}_2$ cannot be the axiom $(L \bot)$, because if it were, then $\nabla^n A$ would have to be $\bot$, making it principal, which is not the case. Third, $\mathcal{D}_2$ cannot be the axiom $(R \top)$, because the left-hand side of this axiom is empty, whereas the antecedent of $\Phi, \nabla^n A \Rightarrow \Lambda$ cannot be empty. Therefore, we conclude that $\mathcal{D}_2$ is not an axiom.

 For the sake of brevity, we will only elaborate on the cases where $\mathcal{D}_2$ ends with the rules $(L \wedge^r)$, $(R \vee_1)$, $(R \!\rightarrow)$, or $(N)$. The remaining cases are either straightforward or analogous.\\

  \noindent $\bullet$ $(L\wedge^r)$: As $\D_2$ ends with $(L \wedge^r)$ and the principal formula is different from $\nabla^n A$, we have $\Phi = \Sigma, \nabla^r (B \wedge C)$, for some multiset $\Sigma$ and some formulas $B$ and $C$. Moreover, $\D_2$ is in the form:
  \begin{prooftree}
    \AXC{$\D_2'$} \noLine
    \UIC{$\Sigma, \nabla^r B, \nabla^r C, \nabla^n A \Rightarrow \Lambda$}
    \RightLabel{$L \wedge ^r$}
    \UIC{$\Sigma, \nabla^r (B \wedge C), \nabla^n A \Rightarrow \Lambda$}
  \end{prooftree}
  By the induction hypothesis on $\D_1$ and $\D_2'$, we have  $\Sigma, \nabla^r B, \nabla^r C, \nabla^n \Pi \Rightarrow \Lambda$. By the rule $(L \wedge^r)$, we have $ \Sigma, \nabla^r (B \wedge C), \nabla^n \Pi \Rightarrow \Lambda$ which is what we wanted.\\

  \noindent $\bullet$ $(R \vee_1$): As $\D_2$ ends with $(R \vee_1)$, we have $\Lambda = \{B \vee C\}$, for some formulas $B$ and $C$, and $\D_2$ is in the form:
  \begin{prooftree}
    \AXC{$\D_2'$} \noLine
    \UIC{$\Phi, \nabla^n A \Rightarrow B$}
    \RightLabel{$R \vee_1$}
    \UIC{$\Phi, \nabla^n A \Rightarrow B \vee C$}
  \end{prooftree}
  Again, use the induction hypothesis for $\D_1$ and $\D_2'$ to get $\Phi, \nabla^n \Pi \Rightarrow B$. Then apply the rule $(R \vee_1)$ to reach the desired sequent.\\

  \noindent $\bullet$ $(R \!\rightarrow$): As $\D_2$ ends with $(R \!\rightarrow)$, we have $\Lambda = \{B \rightarrow C\}$, for some formulas $B$ and $C$, and $\D_2$ has the following form:
  \begin{prooftree}
    \AXC{$\D_2'$} \noLine
    \UIC{$\nabla \Phi, \nabla^{n+1} A, B \Rightarrow C$}
    \RightLabel{$R \!\rightarrow$}
    \UIC{$\Phi, \nabla^n A \Rightarrow B \rightarrow C$}
  \end{prooftree}
  By the induction hypothesis for $\D_1$ and $\D_2'$, we have $\nabla \Phi, \nabla^{n+1} \Pi, B \Rightarrow C$. Then, we can apply the rule $(R \!\rightarrow)$ to get the desired sequent $ \Phi, \nabla^n \Pi \Rightarrow B \rightarrow C$.\\

  \noindent $\bullet$ $(N)$: As $\D_2$ ends with $(N)$, we have $\Phi, \nabla^n A=\nabla \Gamma$ and $\Lambda=\nabla \Delta$, for some multisets $\Gamma$ and $\Delta$ and $\D_2$ has the following form:
  \begin{prooftree}
    \AXC{$\D_2'$} \noLine
    \UIC{$\Gamma \Rightarrow \Delta$}
    \RightLabel{$N$}
    \UIC{$\nabla \Gamma \Rightarrow \nabla \Delta$}
  \end{prooftree}
Note that $\nabla^n A \in \nabla \Gamma$. We claim that $n \geq 1$, because if $n = 0$, then $A \in \nabla \Gamma$, which implies that $A = \nabla B$ for some formula $B$. However, we have already established that this is impossible. Therefore, since $n \geq 1$, it follows that $\nabla^{n-1} A \in \Gamma$. Thus, we can express $\Gamma$ as $\Sigma, \nabla^{n-1} A$, where $\Sigma$ is a multiset. Therefore, $\D'_2$ proves $\Sigma, \nabla^{n-1} A \Rightarrow \Delta$. By the induction hypothesis for $\mathcal{D}_1$ and $\mathcal{D}_2'$, we have $\Sigma, \nabla^{n-1} \Pi \Rightarrow \Delta$. Applying $(N)$, we obtain $\nabla \Sigma, \nabla^n \Pi \Rightarrow \nabla \Delta$, which is the desired sequent.\\

  Part\3. In this part, we assume that $A$ is principal in the last rule of $\mathcal{D}_1$ and $\nabla^n A$ is principal in the last rule of $\mathcal{D}_2$. Clearly, the last rules in both $\mathcal{D}_1$ and $\mathcal{D}_2$ must match in order for $A$ and $\nabla^n A$ to be their principal formulas. Thus, if $R_1$ and $R_2$ are the last rules in $\mathcal{D}_1$ and $\mathcal{D}_2$, the only possible pairs for $(R_1, R_2)$ are $(R \wedge, L \wedge^n)$, $(R \vee_1, L \vee^n)$, $(R \vee_2, L \vee^n)$, $(R \!\supset, L \!\supset^n)$, or, if $n \geq 1$, $(R \!\rightarrow, L \!\rightarrow^{n-1})$. Note that the number $n$ in these rules corresponds to the same number $n$ in $\nabla^n A$. We now investigate each case separately in the following:\\

  \noindent $\bullet$ $(R \wedge, L \wedge ^n)$: As $\D_1$ ends with $(R \wedge)$ and $\D_2$ ends with $(L \wedge^n)$, we have $A=B \wedge C$, for some formulas $B$ and $C$ and $\D_1$ and $\D_2$ are in the form:
  \begin{prooftree}
    \noLine
    \AXC{$\D_1'$}
    \UIC{$\Pi \Rightarrow B$}
    \noLine
    \AXC{$\D_1''$}
    \UIC{$\Pi \Rightarrow C$}
    \RightLabel{$R \wedge$}
    \BIC{$\Pi \Rightarrow B \wedge C$}
    
    \noLine
    \AXC{$\D_2'$}
    \UIC{$\Phi, \nabla^n B, \nabla^n C \Rightarrow \Lambda$}
    \RightLabel{$L \wedge ^n$}
    \UIC{$\Phi, \nabla^n (B \wedge C) \Rightarrow \Lambda$}
    
    \noLine
    \BIC{}
  \end{prooftree}
  Now, consider the following proof-tree:
  \begin{prooftree}
    \AXC{$\D_1''$}\noLine
    \UIC{$\Pi \Rightarrow C$}
    \AXC{$\D_1'$}\noLine
    \UIC{$\Pi \Rightarrow B$}
    \AXC{$\D_2'$}\noLine
    \UIC{$\Phi, \nabla^n B, \nabla^n C \Rightarrow \Lambda$}
    \RightLabel{$\IH_B$}
    \BIC{$\Phi, \nabla^n \Pi, \nabla^n C \Rightarrow \Lambda$}
    \RightLabel{$\IH_C$}
    \BIC{$\Phi, (\nabla^n \Pi)^2 \Rightarrow \Lambda$}
    \RightLabel{$Lc^{(*)}$}\doubleLine
    \UIC{$\Phi, \nabla^n \Pi \Rightarrow \Lambda$}
  \end{prooftree}
  where $Lc^{(*)}$ means a repeated application of the rule $(Lc)$, which is admissible in $\LDL$ by Theorem \ref{thm:lc-adm}.  Notice that both $B$ and $C$ have a lower rank than $A$. Therefore, regardless of the height of the proof-trees, the use of the induction hypothesis is permitted.\\
  
    \noindent $\bullet$ $(R \vee_1, L \vee^n)$: As $\D_1$ ends with $(R \vee_1)$ and $\D_2$ ends with $(L \vee^n)$, we have $A = B \vee C$, for some formulas $B$ and $C$ and $\D_1$ and $\D_2$ are in the form:
  \begin{prooftree}
    \noLine
    \AXC{$\D_1'$}
    \UIC{$\Pi \Rightarrow B$}
    \RightLabel{$R \vee_1$}
    \UIC{$\Pi \Rightarrow B \vee C$}
    \noLine
    \AXC{$\D_2'$}
    \UIC{$\Phi, \nabla^n B \Rightarrow \Lambda$}
    \noLine
    \AXC{$\D_2''$}
    \UIC{$\Phi, \nabla^n C \Rightarrow \Lambda$}
    \RightLabel{$L \vee ^n$}
    \BIC{$\Phi, \nabla^n (B \vee C) \Rightarrow \Lambda$}
    \noLine
    \BIC{}
  \end{prooftree}
  Then, using the induction hypothesis for $B$, we have
  \begin{prooftree}
    \AXC{$\D_1'$}\noLine
    \UIC{$\Pi \Rightarrow B$}
    \AXC{$\D_2'$}\noLine
    \UIC{$\Phi, \nabla^n B \Rightarrow \Lambda$}
    \RightLabel{$\IH_B$}
    \BIC{$\Phi, \nabla^n \Pi \Rightarrow \Lambda$}
  \end{prooftree}
which is what we wanted. The proof of the case $(R \vee_2, L \vee^n)$ is similar.\\
  
  \noindent $\bullet$ $(R \!\rightarrow, L \!\rightarrow ^{n-1})$: As $\D_1$ and $\D_2$ end with $(R \!\rightarrow)$ and $(L \!\rightarrow^{n-1})$, respectively, we have $A = B \rightarrow C$, for some formulas $B$ and $C$, and $\D_1$ and $\D_2$ are in the form:
  { \scriptsize
  \begin{prooftree}
    \noLine
    \AXC{$\D_1'$}
    \UIC{$\nabla \Pi, B \Rightarrow C$}
    \RightLabel{$R \!\rightarrow$}
    \UIC{$\Pi \Rightarrow B \rightarrow C$}
    \noLine
    \AXC{$\D_2'$}
    \UIC{$\Phi, \nabla^n (B \rightarrow C) \Rightarrow \nabla^{n-1} B$}
    \noLine
    \AXC{$\D_2''$}
    \UIC{$\Phi, \nabla^n (B \rightarrow C), \nabla^{n-1} C \Rightarrow \Lambda$}
    \RightLabel{$L \!\rightarrow^{n-1}$}
    \BIC{$\Phi,  \nabla^n (B \rightarrow C) \Rightarrow \Lambda$}
    \noLine
    \BIC{}
  \end{prooftree}
  }
\noindent Consider the following proof-tree and call it $\D$:
  \begin{prooftree}
    \AXC{$\D_1'$} \noLine
    \UIC{$\nabla \Pi, B \Rightarrow C$}
  
    \AXC{$\D_1$} \noLine
    \UIC{$\Pi \Rightarrow B \rightarrow C$}
    \AXC{$\D_2''$} \noLine
    \UIC{$\Phi, \nabla^n (B \rightarrow C), \nabla^{n-1} C \Rightarrow \Lambda$}
    \RightLabel{$\IH_{B \rightarrow C}$}
    \BIC{$\Phi, \nabla^n \Pi, \nabla^{n-1} C \Rightarrow \Lambda$} \RightLabel{$\IH_C$}
    \BIC{$\Phi, (\nabla^n \Pi)^2, \nabla^{n-1} B \Rightarrow \Lambda$}
  \end{prooftree}
Notice that the first use of the induction hypothesis is for $\mathcal{D}_1$ and $\mathcal{D}_2''$, which is permitted because $\mathcal{D}_2''$ has a smaller height. The second use is also allowed since $C$ has a lower rank than $A$. Now, consider the following proof-tree:
  \begin{prooftree}
    \AXC{$\D_1$} \noLine
    \UIC{$\Pi \Rightarrow B \rightarrow C$}
    \AXC{$\D_2'$} \noLine
    \UIC{$\Phi, \nabla^n (B \rightarrow C) \Rightarrow \nabla^{n-1} B$}
    \RightLabel{$\IH_{B \rightarrow C}$}
    \BIC{$\Phi, \nabla^n \Pi \Rightarrow \nabla^{n-1} B$}
  
    \AXC{$\D$}
  
    \RightLabel{$\IH_{\nabla^{n-1} B}$}
    \BIC{$\Phi^2, (\nabla^n \Pi)^3 \Rightarrow \Lambda$}
  
    \RightLabel{$Lc^{(*)}$} \doubleLine
    \UIC{$\Phi, \nabla^{n} \Pi \Rightarrow \Lambda$}
  \end{prooftree}
Here, $Lc^{(*)}$ refers to a repeated application of the rule $(Lc)$. As before, the first application of the induction hypothesis is allowed because the height of $\mathcal{D}_2'$ is smaller than the height of $\mathcal{D}_2$. For the second application, observe that the rank of $\nabla^{n-1} B$ is equal to the rank of $B$, which is smaller than the rank of $A$.\\
  
 \noindent $\bullet$ $(R \!\supset, L \!\supset^n)$: As $\D_1$ and $\D_2$ end with $(R \!\supset)$ and $(L \!\supset^n)$, respectively, we have $A = B \supset C$, for some formulas $B$ and $C$, and $\D_1$ and $\D_2$ are in the form:
  \begin{prooftree}
    \noLine
    \AXC{$\D_1'$}
    \UIC{$\Pi, B \Rightarrow C$}
    \RightLabel{$R \!\supset$}
    \UIC{$\Pi \Rightarrow B \supset C$}
    \noLine
    \AXC{$\D_2'$}
    \UIC{$\Phi, \nabla^n (B \supset C) \Rightarrow \nabla^n B$}
    \noLine
    \AXC{$\D_2''$}
    \UIC{$\Phi, \nabla^n C \Rightarrow \Lambda$}
    \RightLabel{$L \!\supset^n$}
    \BIC{$\Phi,  \nabla^n (B \supset C) \Rightarrow \Lambda$}
    \noLine
    \BIC{}
  \end{prooftree}
  Let $\D$ be the following proof-tree:
  \begin{prooftree}
    \AXC{$\D_1'$} \noLine
    \UIC{$\Pi, B \Rightarrow C$}
  
    \AXC{$\D_2''$} \noLine
    \UIC{$\Phi, \nabla^n C \Rightarrow \Lambda$}
    \RightLabel{$\IH_C$}
    \BIC{$\Phi, \nabla^n \Pi, \nabla^n B \Rightarrow \Lambda$}
  \end{prooftree}
  Now, consider the following proof-tree:
  \begin{prooftree}
    \AXC{$\D_1$} \noLine
    \UIC{$\Pi \Rightarrow B \supset C$}
    \AXC{$\D_2'$} \noLine
    \UIC{$\Phi, \nabla^n (B \supset C) \Rightarrow \nabla^n B$}
    \RightLabel{$\IH_{B \supset C}$}
    \BIC{$\Phi, \nabla^n \Pi \Rightarrow \nabla^n B$}
  
    \AXC{$\D$}
  
    \RightLabel{$\IH_{\nabla^n B}$}
    \BIC{$ \Phi^2, (\nabla^n \Pi)^2 \Rightarrow \Lambda$}
  
    \RightLabel{$Lc^{(*)}$} \doubleLine
    \UIC{$ \Phi, \nabla^n \Pi \Rightarrow \Lambda$}
  \end{prooftree}
where $Lc^{(*)}$ denotes a repeated application of the rule $(Lc)$. This completes the proof for all cases. Thus, we have constructed a proof-tree for the desired sequent for every possible pair of end rules for $\mathcal{D}_1$ and $\mathcal{D}_2$. This concludes the proof of the theorem.
\end{proof}

In the rest of this paper, we utilize the cut-free calculi $\LDL$ and $\LDLS$, designed for the logics $\ldl$ and $\ldls$, respectively, in order to establish certain properties of these logics.

\section{Visser's Rules and Disjunction Property} \label{sec: AdmissibleRules}
Recall that intuitionistic propositional logic has the property that for any finite set $I$, if 
\[
\mathsf{IPC} \vdash [\bigwedge_{i \in I} (A_i \supset B_i)] \supset (C \vee D),
\]
then either $\mathsf{IPC} \vdash [\bigwedge_{i \in I} (A_i \supset B_i)] \supset C$ or $\mathsf{IPC} \vdash [\bigwedge_{i \in I} (A_i \supset B_i)] \supset D$ or $\mathsf{IPC} \vdash [\bigwedge_{i \in I} (A_i \supset B_i)] \supset A_i$, for some $i \in I$. These admissible rules are called \emph{Visser's rules} and play a crucial role, as they provide a basis for all admissible rules of $\mathsf{IPC}$ \cite{Rosalie}. Moreover, note that for the empty $I$, as $\bigwedge \varnothing =\top$, we obtain the well-known \emph{disjunction property} of $\mathsf{IPC}$: if $\mathsf{IPC} \vdash C \vee D$, then either $\mathsf{IPC} \vdash C$ or $\mathsf{IPC} \vdash D$.

In this section, we introduce generalizations of Visser's rules and prove their admissibility in the logics $\ldl$ and $\ldls$. Since the languages $\mathcal{L}$ and $\mathcal{L}_*$ are more expressive than the usual propositional language, we consider three types of Visser's rules: disjunctive, implicative, and Heyting implicative.

\begin{theorem}\label{thm:visser}(Disjunctive Visser's rules for $\ldl$)
  Let $I$ and $J$ be finite sets (possibly empty), $\{m_i\}_{i \in I}$ and $\{n_j\}_{j \in J}$ be sequences of natural numbers and 
  \[
  X=\bigwedge_{i \in I} \nabla^{m_i}(A_i \supset B_i) \wedge \bigwedge_{j \in J} \nabla^{n_j} (C_j \rightarrow D_j).
  \]
  If 
  $
  \ldl \vdash  X \rightarrow E \vee F,
  $
 then one of the following holds:
    \begin{itemize}
      \item[$\bullet$] $\ldl \vdash  X \rightarrow \nabla^{m_i} A_i$, for some $i \in I$,
      \item[$\bullet$] $\ldl \vdash  X \rightarrow \nabla^{n_j-1} C_j$, for some $j \in J$, where $n_j \geq 1$,
      \item[$\bullet$] $\ldl \vdash  X \rightarrow E$, 
      \item[$\bullet$] $\ldl \vdash  X \rightarrow F$.
      \end{itemize}
\end{theorem}
\begin{proof}
First, note that for any multiset $\Gamma$, we have $\LDL \vdash \Gamma \Rightarrow \bigwedge \Gamma$. Therefore, using Corollary \ref{LogicAndCalculi}, Lemma \ref{lem:arrow-bot}, Part $(i)$, and the admissibility of $(cut)$ in $\LDL$, it is enough to prove that for any number $h \in \mathbb{N}$ and any sequent $\Gamma \Rightarrow E \vee F$, where $\Gamma$ has the form $\{\nabla^{m_i}(A_i \supset B_i) \}_{i \in I} \cup \{ \nabla^{n_i} (C_j \rightarrow D_j) \}_{j \in J}$, if $\LDL \vdash_h \Gamma \Rightarrow E \vee F$, then one of the following holds:
\begin{itemize}
    \item[$\bullet$]
 $\LDL \vdash \Gamma \Rightarrow \nabla^{m_i}A_i$, for some $i \in I$,
     \item[$\bullet$]
     $\LDL \vdash \Gamma \Rightarrow \nabla^{n_j-1} C_j$, for some $j \in J$, where $n_j \geq 1$,
    \item[$\bullet$]
    $\LDL \vdash \Gamma \Rightarrow E$, 
     \item[$\bullet$]
     $\LDL \vdash \Gamma \Rightarrow F$.
 \end{itemize}
We prove the claim by induction on $h$. First, note that $h \neq 0$, as $\Gamma \Rightarrow E \vee F$ cannot be an axiom, given its form. Hence, the base case of the induction trivially holds. For the inductive step, let $\D$ be a proof of $\Gamma \Rightarrow E \vee F$ in $\LDL$ with the height at most $h > 0$. Based on the form of formulas in the sequent $\Gamma \Rightarrow E \vee F$, the last rule in $\D$ must be one of the following: $(LW)$, $(Rw)$, $(R \vee_1)$, $(R \vee_2)$, $(L\!\supset^{m_i})$, for some $i \in I$, or $(L\!\rightarrow^{n_j-1})$, for some $j \in J$, where $n_j \geq 1$.

If the last rule in $\mathcal{D}$ is $(LW)$, then there is a proof $\mathcal{D}'$ of height $h' < h$ for the sequent $\Sigma \Rightarrow E \vee F$, where $\Sigma \subseteq \Gamma$. The claim then follows easily by the induction hypothesis for $h'$.  

For $(Rw)$, we have a proof-tree for $(\Gamma \Rightarrow \,)$. Applying $(Rw)$, we can easily prove $\Gamma \Rightarrow E$.

The cases $(R \vee_1)$, $(R \vee_2)$, $(L\!\supset^{m_i})$, and $(L\!\rightarrow^{n_j-1})$ prove the claim directly. We only explain the case $(L\!\rightarrow^{n_j-1})$, as the rest follow similarly.  
Suppose $\mathcal{D}$ ends with $(L\!\rightarrow^{n_j-1})$, for some $j \in J$, where $n_j \neq 0$. Then, $\mathcal{D}$ has the following form:
  \begin{prooftree}
    \AXC{$\Gamma \Rightarrow \nabla^{n_j-1} C_j$}
    \AXC{$\Gamma, \nabla^{n_j-1} D_j \Rightarrow E \vee F$}
    \BIC{$\Gamma \Rightarrow E \vee F$}
  \end{prooftree}
Hence, $\Gamma \Rightarrow \nabla^{n_j-1} C_j$ is provable in $\LDL$.
\end{proof}

\begin{theorem}\label{thm:visserII}(Disjunctive Visser's rules for $\ldls$)
  Let $I$ be a finite set (possibly empty), $\{m_i\}_{i \in I}$  be a sequence of natural numbers and
   \[
  X=\bigwedge_{i \in I} \nabla^{m_i}(A_i \to B_i).
  \]
  If 
  $
  \ldls \vdash  X \rightarrow C \vee D,
  $
 then one of the following holds:
    \begin{itemize}
      \item[$\bullet$] $\ldls \vdash  X \rightarrow \nabla^{m_i-1} A_i$, for some $i \in I$, where $m_i \geq 1$,
      \item[$\bullet$] $\ldls \vdash   X \rightarrow C$, 
      \item[$\bullet$]  $\ldls \vdash  X \rightarrow D$. 
      \end{itemize}
\end{theorem}
\begin{proof}
The proof follows directly from Theorem \ref{thm:visser} and the conservativity of $\ldl$ over $\ldls$, as established in Corollary \ref{LogicAndCalculi}.
\end{proof}

\begin{theorem}[Disjunction property]
	For any formulas $A$ and $B$, if $\ldl \vdash A \vee B$ then either $\ldl \vdash A$ or $\ldl \vdash B$. The same also holds for $\ldls$.
\end{theorem}
\begin{proof}
The proof is similar to  that of Theorem \ref{thm:visser}.
\end{proof}

\begin{theorem}\label{thm:arrow-visser}(Implicative Visser's rules for $\ldl$)
Let $I$ and $J$ be finite sets (possibly empty) and $\{m_i\}_{i \in I}$ and $\{n_j\}_{j \in J}$ be sequences of natural numbers. If 
  \[
  \ldl \vdash [\bigwedge_{i \in I} \nabla^{m_i}(A_i \supset B_i) \wedge \bigwedge_{j \in J} \nabla^{n_j} (C_j \rightarrow D_j)] \rightarrow \nabla^k(E \to F),
  \]
  then one of the following holds:
    \begin{itemize}
      \item[$\bullet$] $\ldl \vdash  [\bigwedge_{i \in I} \nabla^{m_i}(A_i \supset B_i) \wedge \bigwedge_{j \in J} \nabla^{n_j} (C_j \rightarrow D_j)] \rightarrow \nabla^{m_i} A_i$, for some $i \in I$,
      \item[$\bullet$] $\ldl \vdash [\bigwedge_{i \in I} \nabla^{m_i}(A_i \supset B_i) \wedge \bigwedge_{j \in J} \nabla^{n_j} (C_j \rightarrow D_j)] \rightarrow \nabla^{n_j-1} C_j$, for some $j \in J$, where $n_j \geq 1$,
      \item[$\bullet$] $\ldl \vdash [\bigwedge_{i \in I'} \nabla^{m_i-k+1}(A_i \supset B_i) \wedge \bigwedge_{j \in J'} \nabla^{n_j-k+1} (C_j \rightarrow D_j) \wedge E] \rightarrow F$, for some $I' \subseteq I$ and $J' \subseteq J$ such that $m_i, n_j \geq k$, for any $i \in I'$ and $j \in J'$.
      \end{itemize}
\end{theorem}
\begin{proof}
Similar to the proof of Theorem \ref{thm:visser}, we use induction on the height $h$ of a proof-tree $\mathcal{D}$ for any sequent of the form 
\[
\{\nabla^{m_i}(A_i \supset B_i) \}_{i \in I}, \{\nabla^{n_j} (C_j \to D_j) \}_{j \in J} \Rightarrow \nabla^k(E \to F).
\]
First, note that $\mathcal{D}$ cannot be an axiom. Hence, $h \neq 0$, which implies that the base case of the induction trivially holds. 
For the induction step, based on the form of the formulas in the end sequent, the last rule in $\mathcal{D}$ must be either $(LW)$, $(Rw)$, $(N)$, $(L\!\supset^{m_i})$, for some $i \in I$, $(L\!\to^{n_j-1})$, for some $j \in J$ such that $n_j \geq 1$, or $(R\!\to)$ if $k = 0$. The cases $(LW)$, $(Rw)$, $(L\!\supset^{m_i})$, and $(L\!\to^{n_j-1})$ are similar to the corresponding ones in the proof of Theorem \ref{thm:visser}. 

For $(N)$, if the last rule in $\mathcal{D}$ is $(N)$, then $m_i, n_j, k \geq 1$, for any $i \in I$ and $j \in J$. There is a proof-tree $\mathcal{D}'$ of height $h' < h$ for
\[
\{\nabla^{m_i-1}(A_i \supset B_i)\}_{i\in I}, \{\nabla^{n_j-1} (C_j \rightarrow D_j)\}_{j \in J} \Rightarrow \nabla^{k-1}(E \to F)
\]
in $\LDL$. Then, by the induction hypothesis, one of the following holds:
    \begin{itemize}
      \item[$\bullet$] $\LDL \vdash  \{\nabla^{m_i-1}(A_i \supset B_i)\}_{i\in I}, \{\nabla^{n_j-1} (C_j \rightarrow D_j)\}_{j \in J} \Rightarrow \nabla^{m_i-1} A_i$, for some $i \in I$,
      \item[$\bullet$] $\LDL \vdash \{\nabla^{m_i-1}(A_i \supset B_i)\}_{i\in I}, \{\nabla^{n_j-1} (C_j \rightarrow D_j)\}_{j \in J} \Rightarrow \nabla^{n_j-2} C_j$, for some $j \in J$, where $n_j-1 \geq 1$,
      \item[$\bullet$] $\LDL \vdash\{\nabla^{m_i-k+1}(A_i \supset B_i)\}_{i\in I'}, \{\nabla^{n_j-k+1} (C_j \rightarrow D_j)\}_{j \in J'}, E \Rightarrow F$, for some $I' \subseteq I$ and $J' \subseteq J$ such that $m_i-1, n_j-1 \geq k-1$, for any $i \in I'$ and $j \in J'$.
      \end{itemize}
In the first and second cases, it is enough to apply $(N)$. In the last case, there is nothing to prove. Finally, if $k = 0$ and the last rule is $(R\!\to)$, the claim is already proved.
\end{proof}

\begin{theorem}\label{thm:arrow-visserII}(Implicative Visser's rules for $\ldls$)
 Let $I$ be a finite set (possibly empty), and $\{m_i\}_{i \in I}$ be a sequence of natural numbers. If 
  \[
  \ldls \vdash [\bigwedge_{i \in I} \nabla^{m_i} (A_i \rightarrow B_i)] \rightarrow \nabla^k(C \to D),
  \]
  then one of the following holds:
    \begin{itemize}
      \item[$\bullet$] $\ldls \vdash  [\bigwedge_{i \in I} \nabla^{m_i}(A_i \to B_i)] \rightarrow \nabla^{m_i-1} A_i$, for some $i \in I$, where $m_i \geq 1$,
      \item[$\bullet$] $\ldls \vdash [\bigwedge_{i \in I'} \nabla^{m_i-k+1}(A_i \to B_i) \wedge C] \rightarrow D$, for some $I' \subseteq I$ such that $m_i \geq k$, for any $i \in I'$.
      \end{itemize}
\end{theorem}
\begin{proof}
The proof follows directly from Theorem \ref{thm:arrow-visser} and the conservativity of $\ldl$ over $\ldls$, as established in Corollary \ref{LogicAndCalculi}.
\end{proof}

\begin{theorem}\label{thm:arrow-visserIII}(Heyting Implicative Visser's rules for $\ldl$)
Let $I$ and $J$ be finite sets (possibly empty) and $\{m_i\}_{i \in I}$ and $\{n_j\}_{j \in J}$ be sequences of natural numbers. If 
  \[
  \ldl \vdash [\bigwedge_{i \in I} \nabla^{m_i}(A_i \supset B_i) \wedge \bigwedge_{j \in J} \nabla^{n_j} (C_j \rightarrow D_j)] \rightarrow \nabla^k(E \supset F),
  \]
  then one of the following holds:
    \begin{itemize}
      \item[$\bullet$] $\ldl \vdash  [\bigwedge_{i \in I} \nabla^{m_i}(A_i \supset B_i) \wedge \bigwedge_{j \in J} \nabla^{n_j} (C_j \rightarrow D_j)] \rightarrow \nabla^{m_i} A_i$, for some $i \in I$,
      \item[$\bullet$] $\ldl \vdash [\bigwedge_{i \in I} \nabla^{m_i}(A_i \supset B_i) \wedge \bigwedge_{j \in J} \nabla^{n_j} (C_j \rightarrow D_j)] \rightarrow \nabla^{n_j-1} C_j$, for some $j \in J$, where $n_j \geq 1$,
      \item[$\bullet$] $\ldl \vdash [\bigwedge_{i \in I'} \nabla^{m_i-k}(A_i \supset B_i) \wedge \bigwedge_{j \in J'} \nabla^{n_j-k} (C_j \rightarrow D_j) \wedge E] \rightarrow F$, for some $I' \subseteq I$ and $J' \subseteq J$ such that $m_i, n_j \geq k$, for any $i \in I'$ and $j \in J'$.
      \end{itemize}
\end{theorem}
\begin{proof}
The proof is similar to that of Theorem \ref{thm:arrow-visser}.
\end{proof}

\section{Interpolation Theorems}\label{sec: Interpolation}

In this section, we establish the \emph{Craig interpolation theorem} for $\ldl$ using a method similar to that of Maehara \cite{maehara1960interpolation}. We then apply the deduction theorem for $\LDL$ (Theorem \ref{thm:deduction}) to show that the calculus $\LDLP$, which extends $\LDL$ with the $(cut)$ rule, satisfies deductive interpolation. Our approach relies heavily on the Heyting implication, even in the language $\mathcal{L}_*$. Consequently, this method does not extend to proving Craig interpolation for $\ldls$ or deductive interpolation for $\LDLSP$.

\begin{definition}
A logic $L$ has the \emph{Craig Interpolation Property} if for any formulas $A$ and $B$ if $ (A \to B) \in L$, then there exists a formula $C$ such that $(A \to C) \in L$, $(C \to B) \in L$ and $V(C) \subseteq V(A) \cap V(B)$.
\end{definition}

To demonstrate that $\ldl$ satisfies the Craig interpolation property, we establish the following theorem.

\begin{theorem}\label{thm:craig} For any multisets $\Gamma_1$, $\Gamma_2$ and $\Delta$ of formulas, if $\LDL \vdash \Gamma_1 , \Gamma_2 \Rightarrow \Delta$, then there is a formula $C \in \L$ such that:
    \begin{description}
        \item[$(i)$] 
$\LDL \vdash \Gamma_1 \Rightarrow C$, 
        \item[$(ii)$] 
$\LDL \vdash \Gamma_2 , C \Rightarrow \Delta$, and
 \item[$(iii)$] 
$V(C) \subseteq V(\Gamma_1) \cap V(\Gamma_2 \cup \Delta)$.
   \end{description}
Consequently, the logic $\ldl$ enjoys Craig interpolation property.
\end{theorem}
\begin{proof}

Let $\D$ be a proof-tree for $\Gamma_1 , \Gamma_2 \Rightarrow \Delta$ in $\LDL$. By induction on the height of $\D$, we construct a formula $C$ satisfying the properties $(i)$, $(ii)$, and $(iii)$. For the base case, suppose $\D$ is an axiom. If the axiom is $(Id^p)$, there are two cases to consider: either $ \Gamma_1 = \{p\}$ and $\Gamma_2 = \varnothing$, or $\Gamma_1 = \varnothing$ and $\Gamma_2 = \{p\}$. In the first case, set $C = p$, and there is nothing to prove. In the second case, set $C = \top$. Then, $\Gamma_1 \Rightarrow \top$ by $(R \top)$, and $p , \top \Rightarrow p$ by $(Id^p)$ and $(LW)$. Note that $V(\top) = \varnothing$, and hence the condition $(iii)$ is automatic.
If the axiom is $(L \bot)$, again there are two cases: either $\Gamma_1 = \varnothing$ and $\Gamma_2 = \{\bot\}$, or $\Gamma_1 = \{\bot\}$ and $\Gamma_2 = \varnothing$. If $\Gamma_1 = \varnothing$, then take $C = \top$. The conditions $(i)$ and $(iii)$ are trivial, and as $\LDL \vdash \bot, \top \Rightarrow$, we also have $(ii)$. If $\Gamma_1 = {\bot}$, then take $C = \bot$. This time, $(ii)$ and $(iii)$ are trivial, and as $\LDL \vdash \bot \Rightarrow \bot$, we also have $(i)$.
If the axiom is $(R \top)$, then take $C = \top$. Then, $(i)$ and $(iii)$ are trivial. For $(ii)$, we have $\LDL \vdash \top \Rightarrow \top$.

For the inductive step where $\D$ is not an axiom, there are many possibilities for the last rule in $\mathcal{D}$. We explain the cases for rules $(LW)$, $(R \wedge)$,  $(L \vee^n)$, $(L\!\to^n)$, and $(R\!\to)$. The rest are similar:\\

      \item $\bullet$ ($LW$): If the last rule in $\D$ is $(LW)$, then $\Gamma_1 = \Gamma_1', \Sigma_1$ and $\Gamma_2 = \Gamma_2', \Sigma_2$, for some multisets $\Gamma_1'$, $\Gamma_2'$, $\Sigma_1$ and $\Sigma_2$, and the last rule is in the form:
      \begin{prooftree}
      \AXC{$\Gamma_1', \Gamma_2' \Rightarrow \Delta$}
      \RightLabel{$LW$}
      \UIC{$\Gamma_1', \Gamma_2', \Sigma_1, \Sigma_2 \Rightarrow \Delta$}
    \end{prooftree}
   By the induction hypothesis, we have a formula $D$ such that $\Gamma_1' \Rightarrow D$ and $\Gamma_2', D \Rightarrow \Delta$ are provable in $\LDL$, and $V(D) \subseteq V(\Gamma_1') \cap V(\Gamma_2' \cup \Delta)$. Now, set $C = D$. Using $(LW)$ again, it is easy to see that $C$ satisfies all three conditions.\\

    \item[] $\bullet$ ($R\wedge$): Suppose that the last rule of $\D$ is $(R\wedge)$. Then, $\Delta = \{A \wedge B\}$, for some formulas $A$ and $B$, and the last rule in $\D$ is of the following form:
      \begin{prooftree}
      \AXC{$\Gamma_1, \Gamma_2 \Rightarrow A$}
      \AXC{$\Gamma_1, \Gamma_2 \Rightarrow B$}
      \RightLabel{$R\wedge$}
      \BIC{$\Gamma_1, \Gamma_2 \Rightarrow A \wedge B$}
    \end{prooftree}
    By the induction hypothesis, there is a formula $D$ such that $\Gamma_1 \Rightarrow D$ and $\Gamma_2, D \Rightarrow A$ are provable in $\LDL$ and $V(D) \subseteq V(\Gamma_1) \cap V(\Gamma_2 \cup \{A\})$. Similarly, there is a formula $E$ such that $\Gamma_1 \Rightarrow E$ and $\Gamma_2, E \Rightarrow B$ are provable in $\LDL$ and $V(E) \subseteq V(\Gamma_1) \cap V(\Gamma_2 \cup \{B\})$. 
    Define $C = D \wedge E$. For $(i)$, we clearly have $\Gamma_1 \Rightarrow D \wedge E$. For $(ii)$, consider the following proof-tree:
       \begin{prooftree}
      \AXC{$\Gamma_2, D \Rightarrow A$}
      \RightLabel{$LW$}
      \UIC{$\Gamma_2, D, E \Rightarrow A$}
      \RightLabel{$L\wedge^0$}
      \UIC{$\Gamma_2, D \wedge E \Rightarrow A$}
       \AXC{$\Gamma_2, E \Rightarrow B$}
      \RightLabel{$LW$}
      \UIC{$\Gamma_2, D, E \Rightarrow B$}
      \RightLabel{$L\wedge^0$}
      \UIC{$\Gamma_2, D \wedge E \Rightarrow B$}
      \RightLabel{$R\wedge$}
      \BIC{$\Gamma_2, D \wedge E \Rightarrow A \wedge B$}
    \end{prooftree}
Checking the condition $(iii)$ is straightforward.\\
    
      \item[$\bullet$] ($L \vee^n$):
      Suppose the last rule of $\D$ is $(L \vee^n)$. Then,  the principal formula of the last rule is $\nabla^n(A \vee B)$, for some formulas $A$ and $B$. Now,
      there are two cases to consider: either $\nabla^n (A \vee B) \in \Gamma_1$ or $\nabla^n (A \vee B) \in \Gamma_2$. In the first case, $\Gamma_1 = \Gamma_1', \nabla^n(A \vee B)$, for some multiset $\Gamma_1'$ and the last rule is in the form:
      \begin{prooftree}
        \AXC{$\Gamma_1', \Gamma_2, \nabla^n A \Rightarrow \Delta$}
         \AXC{$\Gamma_1', \Gamma_2, \nabla^n B \Rightarrow \Delta$}
        \RightLabel{$L \vee^n$}
        \BIC{$\Gamma_1', \Gamma_2 , \nabla^n (A \vee B) \Rightarrow \Delta$}
      \end{prooftree}
      By the induction hypothesis, there is a formula $D$ such that $\Gamma_1', \nabla^n A \Rightarrow D$ and $\Gamma_2, D  \Rightarrow \Delta$ are provable in $\LDL$ and $V(D) \subseteq V(\Gamma_1' \cup \{\nabla^n A\}) \cap V(\Gamma_2 \cup \Delta)$.
      Similarly, there is a formula $E$ such that $\Gamma_1', \nabla^n B \Rightarrow E$ and $\Gamma_2, E  \Rightarrow \Delta$ are provable in $\LDL$ and $V(E) \subseteq V(\Gamma_1' \cup \{\nabla^n B\}) \cap V(\Gamma_2 \cup \Delta)$.
      Take $C = D \vee E$. For $(i)$, consider the following proof-tree:
       \begin{prooftree}
        \AXC{$\Gamma_1', \nabla^n A \Rightarrow D$}
         \RightLabel{$R\vee_1$}
        \UIC{$\Gamma_1', \nabla^n A \Rightarrow D \vee E$}
         \AXC{$\Gamma_1', \nabla^n B \Rightarrow E$}
          \RightLabel{$R\vee_2$}
        \UIC{$\Gamma_1', \nabla^n B \Rightarrow D \vee E$}
        \RightLabel{$L \vee^n$}
        \BIC{$\Gamma_1', \nabla^n (A \vee B) \Rightarrow D \vee E$}
      \end{prooftree}
      For $(ii)$, it is enough to apply the rule $(L \vee^0)$ on the sequent $\Gamma_2, D \Rightarrow \Delta$ and $\Gamma_2, E \Rightarrow \Delta$ to get $\Gamma_2, D \vee E \Rightarrow \Delta$. Checking $(iii)$ is straightforward.  
  
      In the second case, $\Gamma_2 = \Gamma_2', \nabla^n(A \vee B)$, for some multiset $\Gamma_2'$ and the last rule is in the form:
       \begin{prooftree}
        \AXC{$\Gamma_1, \Gamma'_2, \nabla^n A \Rightarrow \Delta$}
         \AXC{$\Gamma_1, \Gamma'_2, \nabla^n B \Rightarrow \Delta$}
        \RightLabel{$L \vee^n$}
        \BIC{$\Gamma_1, \Gamma'_2 , \nabla^n (A \vee B) \Rightarrow \Delta$}
      \end{prooftree}
      By the induction hypothesis, there is a formula $D$ such that $\Gamma_1 \Rightarrow D$ and $\Gamma'_2, \nabla^n A, D \Rightarrow \Delta$ are provable in $\LDL$ and $V(D) \subseteq V(\Gamma_1) \cap V(\Gamma'_2 \cup \{\nabla^n A\} \cup \Delta)$.
      Similarly,  there is a formula $E$ such that $\Gamma_1 \Rightarrow E$ and $\Gamma'_2, \nabla^n B, E \Rightarrow \Delta$ are provable in $\LDL$ and $V(E) \subseteq V(\Gamma_1) \cap V(\Gamma'_2 \cup \{\nabla^n B\} \cup \Delta)$.
      Take $C = D \wedge E$. Then, $(i)$ and $(iii)$ are clear. For $(ii)$, consider the following proof-tree:
        \begin{prooftree}
        \AXC{$\Gamma_2', \nabla^n A, D \Rightarrow \Delta$}
         \RightLabel{$LW$}
        \UIC{$\Gamma_2', \nabla^n A, D, E \Rightarrow \Delta$}
          \RightLabel{$L\wedge^0$}
        \UIC{$\Gamma_2', \nabla^n A, D \wedge E \Rightarrow \Delta$}
         \AXC{$\Gamma_2', \nabla^n B, E \Rightarrow \Delta$}
         \RightLabel{$LW$}
        \UIC{$\Gamma_2', \nabla^n B, D, E \Rightarrow \Delta$}
          \RightLabel{$L\wedge^0$}
        \UIC{$\Gamma_2', \nabla^n B, D \wedge E \Rightarrow \Delta$}
        \RightLabel{$L \vee^n$}
        \BIC{$\Gamma_2', \nabla^n (A \vee B), D \wedge E \Rightarrow \Delta$}
      \end{prooftree}

      \item[] $\bullet$ ($L\!\rightarrow^n$): Suppose that the last rule of $\D$ is $(L\! \rightarrow^n)$. Then, the principal formula of the last rule is $\nabla^{n+1}(A \to B)$, for some formulas $A$ and $B$. Now,
      there are two cases to consider: either $\nabla^{n+1}(A \to B) \in \Gamma_1$ or $\nabla^{n+1}(A \to B) \in \Gamma_2$.
      For the first case, $\Gamma_1 = \Gamma_1', \nabla^{n+1} (A \rightarrow B)$, for some multiset $\Gamma_1'$ and the last rule in $\D$ must be of the form:
      \begin{prooftree}
        \AXC{$\Gamma_1', \Gamma_2, \nabla^{n+1} (A \rightarrow B) \Rightarrow \nabla^n A$}
        \AXC{$\Gamma_1', \Gamma_2, \nabla^{n+1} (A \rightarrow B), \nabla^n B \Rightarrow \Delta$}
        \RightLabel{\footnotesize $L\!\rightarrow^n$}
        \BIC{$\Gamma_1', \Gamma_2 , \nabla^{n+1} (A \rightarrow B) \Rightarrow \Delta$}
      \end{prooftree}
      By the induction hypothesis, there is a formula $D$ such that $\Gamma_2 \Rightarrow D$ and $\Gamma'_1, \nabla^{n+1} (A \rightarrow B), D \Rightarrow \nabla^n A$ are provable in $\LDL$ and 
      \[
      V(D) \subseteq V(\Gamma'_1 \cup \{\nabla^{n+1}(A \to B), \nabla^n A\}) \cap V(\Gamma_2).
      \]
      Again, by the induction hypothesis, there is a formula $E$ such that the sequents  $\Gamma'_1, \nabla^{n+1}(A \to B), 
      \nabla^n B \Rightarrow E$ and $\Gamma_2, E \Rightarrow \Delta$ are provable in $\LDL$ and 
      \[
      V(E) \subseteq V(\Gamma'_1 \cup \{\nabla^{n+1}(A \to B), \nabla^n B\}) \cap V(\Gamma_2 \cup \Delta).
      \]
      Define $C = D \supset E$. For $(i)$, consider the following proof-tree:
      \begin{prooftree}
        \AXC{$\Gamma_1', \nabla^{n+1} (A \rightarrow B), D \Rightarrow \nabla^n A$}
        \AXC{$\Gamma_1', \nabla^{n+1} (A \rightarrow B), \nabla^n B \Rightarrow E$}
        \RightLabel{\footnotesize $LW$}
        \UIC{$\Gamma_1', \nabla^{n+1} (A \rightarrow B), D, \nabla^n B \Rightarrow E$}
        \RightLabel{\footnotesize $L\!\rightarrow^n$}
        \BIC{$\Gamma_1', \nabla^{n+1} (A \rightarrow B), D \Rightarrow E$}
         \RightLabel{\footnotesize $R\!\supset$}
        \UIC{$\Gamma_1', \nabla^{n+1} (A \rightarrow B) \Rightarrow D \supset E$}
      \end{prooftree}
     For $(ii)$, consider the following proof-tree:
      \begin{prooftree}
        \AXC{$\Gamma_2 \Rightarrow D$}
        \RightLabel{\footnotesize $LW$}
        \UIC{$\Gamma_2, D \supset E \Rightarrow D$}
        \AXC{$\Gamma_2, E  \Rightarrow \Delta$}
        \RightLabel{\footnotesize $L\!\supset^0$}
        \BIC{$\Gamma_2, D \supset E  \Rightarrow \Delta$}
      \end{prooftree}
     The condition $(iii)$ is easy to check.

     For the second case, $\Gamma_2 = \Gamma_2', \nabla^{n+1} (A \rightarrow B)$, for some multiset $\Gamma_2'$ and the last rule in $\D$ must be of the form:
      \begin{prooftree}
        \AXC{$\Gamma_1, \Gamma_2', \nabla^{n+1} (A \rightarrow B) \Rightarrow \nabla^n A$}
        \AXC{$\Gamma_1, \Gamma_2', \nabla^{n+1} (A \rightarrow B), \nabla^n B \Rightarrow \Delta$}
        \RightLabel{\footnotesize $L\!\rightarrow^n$}
        \BIC{$\Gamma_1, \Gamma_2' , \nabla^{n+1} (A \rightarrow B) \Rightarrow \Delta$}
      \end{prooftree}
      By the induction hypothesis, there is a formula $D$ such that $\Gamma_1 \Rightarrow D$ and $\Gamma'_2, \nabla^{n+1} (A \rightarrow B), D \Rightarrow \nabla^n A$ are provable in $\LDL$ and
       \[
      V(D) \subseteq V(\Gamma_1) \cap V(\Gamma'_2 \cup \{\nabla^{n+1}(A \to B), \nabla^n A\}).
      \]
      Again, by the induction hypothesis, there is a formula $E$ such that $\Gamma_1 \Rightarrow E$ and $\Gamma'_2, \nabla^{n+1} (A \rightarrow B), \nabla^n B, E \Rightarrow \Delta$ are provable in $\LDL$ and
     \[
      V(E) \subseteq V(\Gamma_1) \cap V(\Gamma'_2 \cup \{\nabla^{n+1}(A \to B), \nabla^n B\} \cup \Delta).
      \]
    Define $C = D \wedge E$. For $(i)$, by applying the rule $(R \wedge)$, we have $\Gamma_1 \Rightarrow D \wedge E$. For $(ii)$, consider the following proof-tree:
      {\footnotesize
      \begin{prooftree}
        \AXC{$\Gamma'_2, \nabla^{n+1}(A \to B), D \Rightarrow \nabla^n A$}
        \RightLabel{\footnotesize $LW$}
        \UIC{$\Gamma'_2, \nabla^{n+1}(A \to B), D, E \Rightarrow \nabla^n A$}
        \RightLabel{\footnotesize $L\wedge^0$}
        \UIC{$\Gamma'_2, \nabla^{n+1}(A \to B), D \wedge E \Rightarrow \nabla^n A$}
        \AXC{$\Gamma'_2, \nabla^{n+1}(A \to B), E, \nabla^n B \Rightarrow \Delta$}
        \RightLabel{\footnotesize $LW$}
        \UIC{$\Gamma'_2, \nabla^{n+1}(A \to B), D, E, \nabla^n B \Rightarrow \Delta$}
        \RightLabel{\footnotesize $L\wedge^0$}
        \UIC{$\Gamma'_2, \nabla^{n+1}(A \to B), D \wedge E, \nabla^n B \Rightarrow \Delta$}
        \RightLabel{\footnotesize $L\!\to^n$}
        \BIC{$\Gamma'_2, \nabla^{n+1}(A \to B), D \wedge E  \Rightarrow \Delta$}
      \end{prooftree}
      }
     The condition $(iii)$ is easy to check.\\
  
      \item[] $\bullet$ ($R\!\rightarrow$): Suppose the last rule of $\D$ is $(R\!\rightarrow)$. Then, $\Delta = \{A \rightarrow B\}$, for some formulas $A$ and $B$, and the last rule in $\D$ is of the following form:
      \begin{prooftree}
      \AXC{$\nabla \Gamma_1, \nabla \Gamma_2, A \Rightarrow B$}
      \RightLabel{$R\!\rightarrow$}
      \UIC{$\Gamma_1, \Gamma_2 \Rightarrow A \rightarrow B$}
    \end{prooftree}
    By the induction hypothesis, there is a formula $D$ such that $\nabla \Gamma_1 \Rightarrow D$ and $\nabla \Gamma_2, A, D \Rightarrow B$ are provable in $\LDL$ and $V(D) \subseteq V(\nabla \Gamma_1) \cap V(\nabla \Gamma_2 \cup \{A, B\})$. Set $C = \top \rightarrow D$. For $(i)$ and $(ii)$, consider the following proof-trees:
\begin{center}
 \begin{tabular}{c c}
\AxiomC{$\nabla \Gamma_1 \Rightarrow D$}
 \RightLabel{$LW$}
 \UnaryInfC{$\nabla \Gamma_1, \top \Rightarrow D$}
 \RightLabel{$R\!\rightarrow$} 
 \UnaryInfC{$\Gamma_1 \Rightarrow \top \rightarrow D$}
 \DisplayProof
 & \;\;\;\;\;
 \AxiomC{$\nabla \Gamma_2, D, A \Rightarrow B$}
 \RightLabel{$(*)$} 
 \UnaryInfC{$\nabla \Gamma_2, \nabla(\top \rightarrow D), A \Rightarrow B$}
 \RightLabel{$R\!\to$} 
 \UnaryInfC{ $\Gamma_2, \top \rightarrow D \Rightarrow A \rightarrow B$,}
 \DisplayProof
\end{tabular}
\end{center} 
where $(*)$ means an application of Lemma \ref{lem:id-adm}, Part $(iii)$. The condition $(iii)$ is easy to check.

This completes the inductive proof of the main claim. Finally, to prove Craig interpolation for $\ldl$, if $\ldl \vdash A \to B$, then $\LDL \vdash \, \Rightarrow A \to B$, by Corollary \ref{LogicAndCalculi}. By Lemma \ref{lem:arrow-bot}, Part $(i)$, we have $\LDL \vdash A \Rightarrow B$. Therefore, by the main claim, there is a formula $C$ such that $\LDL \vdash A \Rightarrow C$, $\LDL \vdash C \Rightarrow B$ and $V(C) \subset V(A) \cap V(B)$. Using a similar argument as above, it is easy to see that $\ldl \vdash A \to C$ and $\ldl \vdash C \to B$.
  \end{proof}

\begin{remark}
Notice that the proof for Theorem \ref{thm:craig} does not work for $\LDLS$, since we are using $\supset$ in the interpolant formula to handle the case $(L\!\to^n)$.
\end{remark}

We now use the deduction theorem, Theorem \ref{thm:deduction}, and Theorem \ref{thm:craig} to prove that $\LDLP$ enjoys the deductive interpolation.

\begin{theorem}[Deductive Interpolation for $\LDL$]\label{thm:deductive}
  For any formulas $A$ and $B$, if $\Rightarrow A \vdash_{\LDLP}\ \Rightarrow B$ then there exists a formula $C$ such that:
	\begin{itemize}
          \item[$\bullet$] 
$\Rightarrow A \vdash_{\LDLP} \, \Rightarrow C$
        \item[$\bullet$] 
$\Rightarrow C \vdash_{\LDLP} \, \Rightarrow B$ and
 \item[$\bullet$] 
$V(C) \subseteq V(A) \cap V(B)$.
	\end{itemize}
\end{theorem}
\begin{proof}
Let $\Rightarrow A \vdash_{\LDLP} \, \Rightarrow B$. By deduction theorem, Theorem \ref{thm:deduction}, there is a set $\Sigma_A$ of the variants of $A$ such that $\LDL \vdash \Sigma_A \Rightarrow B$. By Theorem \ref{thm:craig} for $\Gamma_1 = \Sigma_A$, $\Gamma_2 = \varnothing$ and $\Delta =\{B\}$, we get a formula $C$ such that $\LDL \vdash \Sigma_A \Rightarrow C$, $\LDL \vdash C \Rightarrow B$ and $V(C) \subseteq V(\Sigma_A) \cap V(B)$. Using Theorem \ref{thm:deduction} again,
it is easy to see that $\Rightarrow A \vdash_{\LDLP} \, \Rightarrow C$ and
$\Rightarrow C \vdash_{\LDLP} \, \Rightarrow B$. Moreover, as 
$V(C) \subseteq V(\Sigma_A) \cap V(B)$ and $V(\Sigma_A) \subseteq V(A)$, we reach $V(C) \subseteq V(A) \cap V(B)$.
\end{proof}

\bibliographystyle{plain}
\bibliography{Proof}

\end{document}